\documentclass{cmslatex}
          
\usepackage[paperwidth=7in, paperheight=10in, margin=.875in]{geometry}

\usepackage{amsmath,amsfonts,amssymb}
\usepackage{hyperref}
\usepackage{lineno}
\usepackage{color}
\usepackage{graphicx,epsfig}
\usepackage[percent]{overpic}
\usepackage{bbm}
\usepackage{fourier-orns}

\newcommand{\R}{\mathbb{R}}
\newcommand{\fmax}{f_{\text{max}}}
\newcommand{\edge}{\textsc{e}}
\newcommand{\vertex}{\textsc{v}}
\newcommand{\Dx}{\Delta x}
\newcommand{\Dt}{\Delta t}
\newcommand{\demand}{\textsc{d}}
\newcommand{\supply}{\textsc{s}}
\newcommand{\Nc}{J}
\newcommand{\Nce}{J_{\edge}}
\newcommand{\Nceuno}{J_{\edge_1}}
\newcommand{\Ncedue}{J_{\edge_2}}
\newcommand{\NT}{N_T}
\newcommand{\Nnl}{\ell}
\newcommand{\Np}{n_{p}}
\newcommand{\Nvp}{n^\vertex_{p}}
\newcommand{\Nvinc}{n^\vertex_\text{inc}}
\newcommand{\Nvout}{n^\vertex_\text{out}}
\newcommand{\GW}{\mathcal G_{\mathcal N}^\Delta}
\newcommand{\mis}[1]{\nu^{#1}}
\newcommand{\einc}{\edge}
\newcommand{\eout}{\edge^\prime}
\newcommand{\Hnorm}{\hat{\mathcal{H}}}
\newcommand{\REV}[1]{\textcolor{black}{#1}}

\renewcommand{\theequation}{\arabic{section}.\arabic{equation}}

\newtheorem{thm}{Theorem}[section]
\newtheorem{prop}[thm]{Proposition}

\newtheorem{defn}[thm]{Definition}

\newtheorem{rem}[thm]{Remark}

\begin{document}

\title{Sensitivity analysis of the LWR model \\ for traffic forecast on large networks \\ using Wasserstein distance
}


\author{Maya Briani\thanks{
Istituto per le Applicazioni del Calcolo ``M.\ Picone'', 
Consiglio Nazionale delle Ricerche,
Via dei Taurini, 19 -- Rome, Italy, 
m.briani@iac.cnr.it}
\and{Emiliano Cristiani\thanks{
Istituto per le Applicazioni del Calcolo ``M.\ Picone'', 
Consiglio Nazionale delle Ricerche,
Via dei Taurini, 19 -- Rome, Italy, 
e.cristiani@iac.cnr.it}}\and{\!\!\!
Elisa Iacomini}\thanks{
Dipartimento di Scienze di Base e Applicate per l'Ingegneria, 
Sapienza -- Universit\`a di Roma, 
Via Scarpa, 14/16 -- Rome, Italy, elisa.iacomini@sbai.uniroma1.it}}




\pagestyle{myheadings} 

\markboth{Sensitivity analysis of the LWR model for traffic forecast on large networks}{M. Briani, E. Cristiani, and E. Iacomini} 

\maketitle

\begin{abstract}
In this paper we investigate the sensitivity of the LWR model on network to its parameters and to the network itself. The quantification of sensitivity is obtained by measuring the Wasserstein distance between two LWR solutions corresponding to different inputs. To this end, we propose a numerical method to approximate the Wasserstein distance between two density distributions defined on a network. 
We found a large sensitivity to the traffic distribution at junctions, the network size, and the network topology.
\end{abstract}
          



\section{Introduction.}\label{sec:intro}	
In this paper we deal with the sensitivity analysis of the celebrated LWR model for traffic forecast on networks. The LWR model was introduced by Lighthill and Whitham \cite{LW} and Richards \cite{R}, and it paved the way to the macroscopic description of traffic flow \cite{bellomo2011SR,habermanbook,helbing2001RMP}. 
In a macroscopic model, traffic is described in terms of average density $\rho=\rho(x, t)$ and velocity $v=v(x,t)$ of vehicles, rather tracking each single vehicle. 
The natural assumption that the total mass is conserved along the road leads to impose that $\rho$ and $v$ obey
\begin{equation}\label{LWR}
\partial_t\rho+\partial_x(\rho v)=0,\qquad \rho(x,0)=\rho^0(x)
\end{equation}
for $(x,t)\in\R\times[0,T]$, a final time $T>0$, and initial distribution $\rho^0$. In first-order models like LWR, the velocity $v=v(\rho)$ is given as  function of the density $\rho$, thus closing the equation. \REV{Without loss of generality,} in the rest of the paper we will assume that the maximal density of vehicles is normalized to 1 and the fundamental diagram $f(\rho):=\rho\ v(\rho)$ has the form
\begin{equation}\label{FD}
f(\rho):=\left\{
\begin{array}{ll}
\frac{\fmax}{\sigma}\rho,  & \text{ if } \rho\leq\sigma \\ [2mm]
\frac{\fmax}{\sigma-1}(\rho-1), & \text{ if } \rho>\sigma \\
\end{array}
\right.
\end{equation}
for some $\sigma\in(0,1)$ and $\fmax>0$, see Fig.\ \ref{fig:FD}. 
\begin{figure}[h!]
\centerline{
\begin{overpic}[width=0.49\textwidth]{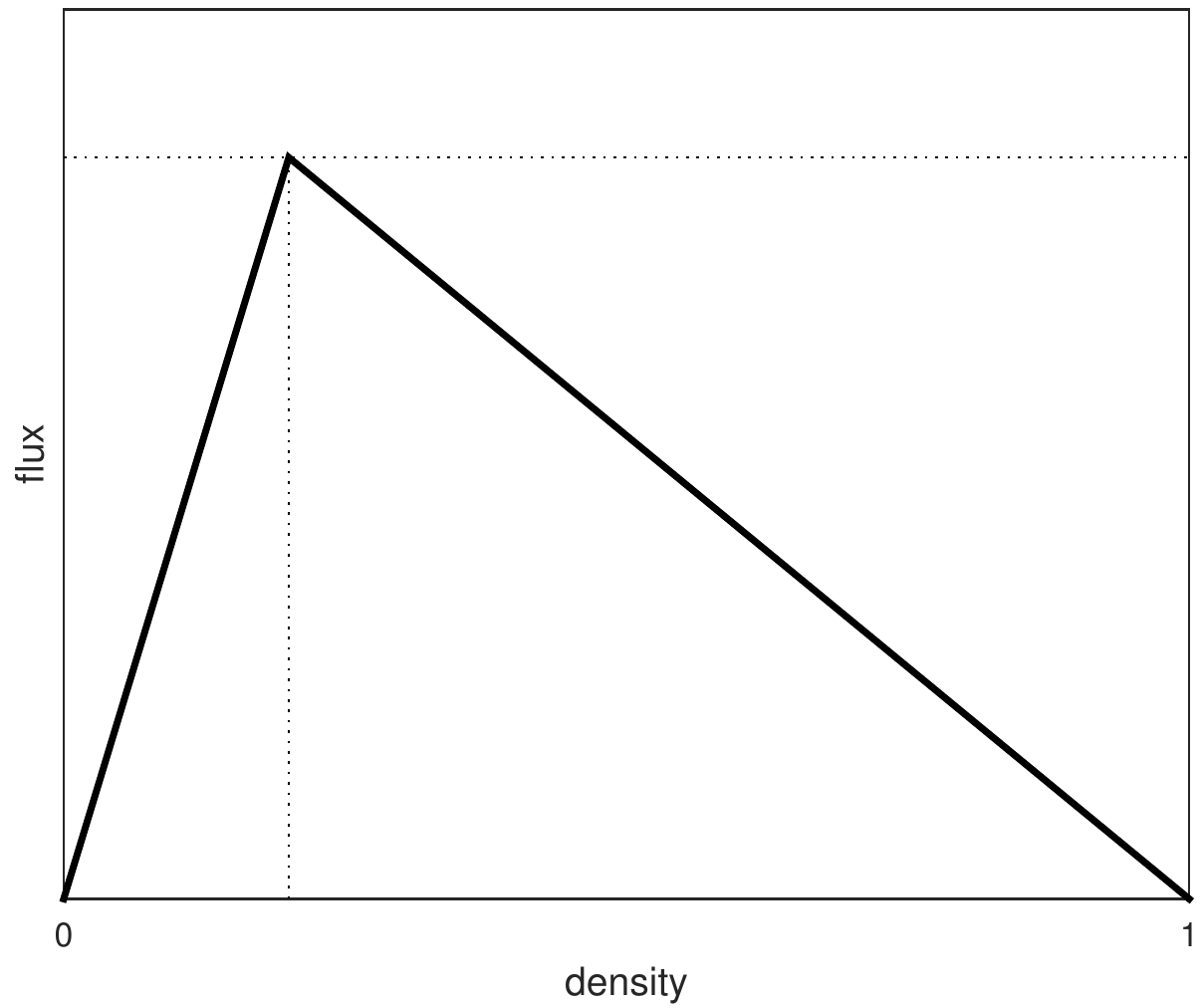}
\put(23,4){$\sigma$}\put(40,73){$\fmax$}\end{overpic}} 
\caption{Fundamental diagram $f(\rho)=\rho\ v(\rho)$.}
\label{fig:FD}
\end{figure}

In order to describe real situations where the vehicles move on several interconnected roads, the simple model \eqref{LWR} is not sufficient. This has motivated several authors to consider analogous equations on a network (metric graph), which is a directed graph whose edges are equipped with a system of coordinates. In this context vertexes are called \emph{junctions} and edges are called \emph{roads}.
The natural way to extend \eqref{LWR} to a network is to assume that the conservation law is separately satisfied on each road for all times $t>0$. Moreover, additional conditions have to be imposed at junctions, because in general the conservation of the mass alone is not sufficient to characterize a unique solution. We refer the reader to the book by Garavello and Piccoli \cite{piccolibook} for more details about the ill-posedness of the problem at junctions. Multiple workarounds for such ill-posedness have been suggested in the literature: 
(i) maximization of the flux across junctions and introduction of priorities among the incoming roads \cite{coclite2005SIMA,piccolibook,holden1995SIMA}; 
(ii) introduction of a \emph{buffer} to model the junctions by means of additional ODE coupled with \eqref{LWR} \cite{bressan2015NHM,garavello2012DCDS-A,garavello2013bookchapt,herty2009NHM}; 
(iii) reformulation of the problem on all possible paths on the network rather than on roads and junctions \cite{bretti2014DCDS-S,briani2014NHM,hilliges1995TRB}.
In general, they all allow to determine a unique solution for the traffic evolution on the network, but the solutions might be different.

In the following we will employ the approach (iii) together with a Godunov-based numerical approximation \cite{bretti2014DCDS-S,briani2014NHM}, also used to reproduce Wardrop equilibria \cite{cristiani2015NHM}. 
More recently, it was shown that the solution computed by the path-based model coincides with the many-particle limit of the first-order microscopic follow-the-leader model \cite{cristiani2016NHM}.

\medskip

In this paper we aim at quantifying the sensitivity of the LWR model to its inputs. In particular we will focus on the following inputs:
\begin{itemize}
\item initial datum;
\item fundamental diagram;
\item traffic distribution at junctions;
\item network size;
\item network connectivity.
\end{itemize}

The sensitivity will be measured by computing the distance between two LWR solutions obtained with different inputs or parameters, in order to understand their impact on the final solution. The question arises which distance is more suitable to this kind of investigation. It is by now well understood that $L^p$ distances do not catch the natural concept of distance among densities, see, e.g., the discussion in \cite[Sec.\ 7.1]{cristianibook} and Sec.\ \ref{sec:wassersteindistance} in this paper.
A notion of distance which instead appears more natural is that of \emph{Wasserstein distance}, see, e.g., \cite{bolley2005JHDE,cristianibook,difrancesco2015ARMA,fjordholm2016SINUM,nessyahu1992SINUM}. 
The bottleneck for using this distance on networks comes from the computational side. In fact, the classical definition of Wasserstein distance is not suitable for numerical approximation. 
Recent characterizations also seem to be unfit for this goal. Let us mention in this regard the variational approach proposed by Maz\'on et al.\ \cite{mazon2015SIOPT}, which generalizes to networks the results by Evans and Gangbo \cite{evansgangbo}. It can be also shown \cite{mazon2015SIOPT} that the Wasserstein distance has a nice link with the $p$-Laplacian operator. This leads naturally to a PDE-based approach to solve the problem. Unfortunately, preliminary investigations \cite{iacominitesi} have shown that both these approaches are highly ill-conditioned and then computationally infeasible.
Let us finally mention the semi-discrete approach proposed by Treleaven and Frazzoli \cite{treleaven2014ACC}, which is well-conditioned but in practice it is still too expensive from the computationally point of view, especially on large networks. It is therefore necessary to find an alternative method to compute the Wasserstein distance.

Following the lines of \cite[Chap.\ 19]{sinhabook}, we propose and implement a pure discrete, reasonably fast algorithm to approximate the Wasserstein distance on a network, based on a linear programming method.
To the authors' knowledge, the study presented here is the first one employing the Wasserstein distance to quantify the sensitivity of a macroscopic model for traffic flow on networks.

\subsubsection*{Paper organization}
In Sec.\ \ref{sec:LWR} we present in detail the mathematical model, which is a local version of the path-based approach \cite{bretti2014DCDS-S,briani2014NHM}, suitable to deal with large networks.
In Sec.\ \ref{sec:wassersteindistance} we motivate the choice of the Wasserstein distance for the subsequent sensitivity analysis and we describe the numerical procedure to compute it.
In Sec.\ \ref{sec:sensitivity}, which is the core of the paper, we present several tests designed to show the sensitivity of the model to various parameters and inputs.
Finally, in Sec.\ \ref{sec:conclusions} we sketch some conclusions.


\section{A path-based model for large networks.}\label{sec:LWR}
In this section we introduce the network and the model we will use in the following to describe the traffic flow.

\subsection{Basic definitions: graph, network, and computational grid.}\label{sec:basicdefs}
Let us consider a connected and directed graph $\mathcal G=(\mathcal E,\mathcal V)$ consisting of a finite set of vertexes $\mathcal V$ and a set of oriented edges $\mathcal E$ connecting the vertexes.

Stemming from $\mathcal G$, we build the network $\mathcal N$ by assigning to each edge $\edge\in\mathcal E$ a positive length $L_\edge\in(0,+\infty)$. Moreover, a coordinate is assigned to each point of the edge. The coordinate will be denoted by $x_\edge$ and it increases according to the direction of the edge, i.e.\ we have $x_\edge=0$ at the initial vertex and $x_\edge=L_\edge$ at the terminal vertex.

For numerical purposes, the network has to be discretized by means of a grid. Let us denote by $\Dx$ the space step.  
To avoid technicalities, let us assume that the length of all the edges is a multiple of $\Dx$ so that we can easily use the same grid size everywhere in $\mathcal N$. In this way, each edge is being divided in $\Nce:=\frac{L_\edge}{\Dx}$ cells. The total number of cells in $\mathcal N$ is $\Nc:=\sum_{\edge\in\mathcal E}\Nce$. The center of each cell on edge $\edge$ will be denoted by $x_{\edge,j}$, $j=1,\ldots,\Nce$, and the cell itself by $C_{\edge,j}=\big[x_{\edge,j-\frac12},x_{\edge,j+\frac12}\big)$.

Finally, let us denote by $\Dt$ the time step. The number of time iterations will be denoted by $\NT:=\frac{T}{\Dt}$.

\subsection{The model.}\label{sec:themodel}
In \cite{bretti2014DCDS-S,briani2014NHM} a theoretical framework and a Godunov-based numerical algorithm was introduced to solve the LWR equations on small networks. The problem is reformulated on (possibly overlapping) paths joining all possible sources with all possible destinations, see Fig.\ \ref{fig:paths}. 
\begin{figure}[h!]
\centerline{
\begin{overpic}[width=0.49\textwidth]{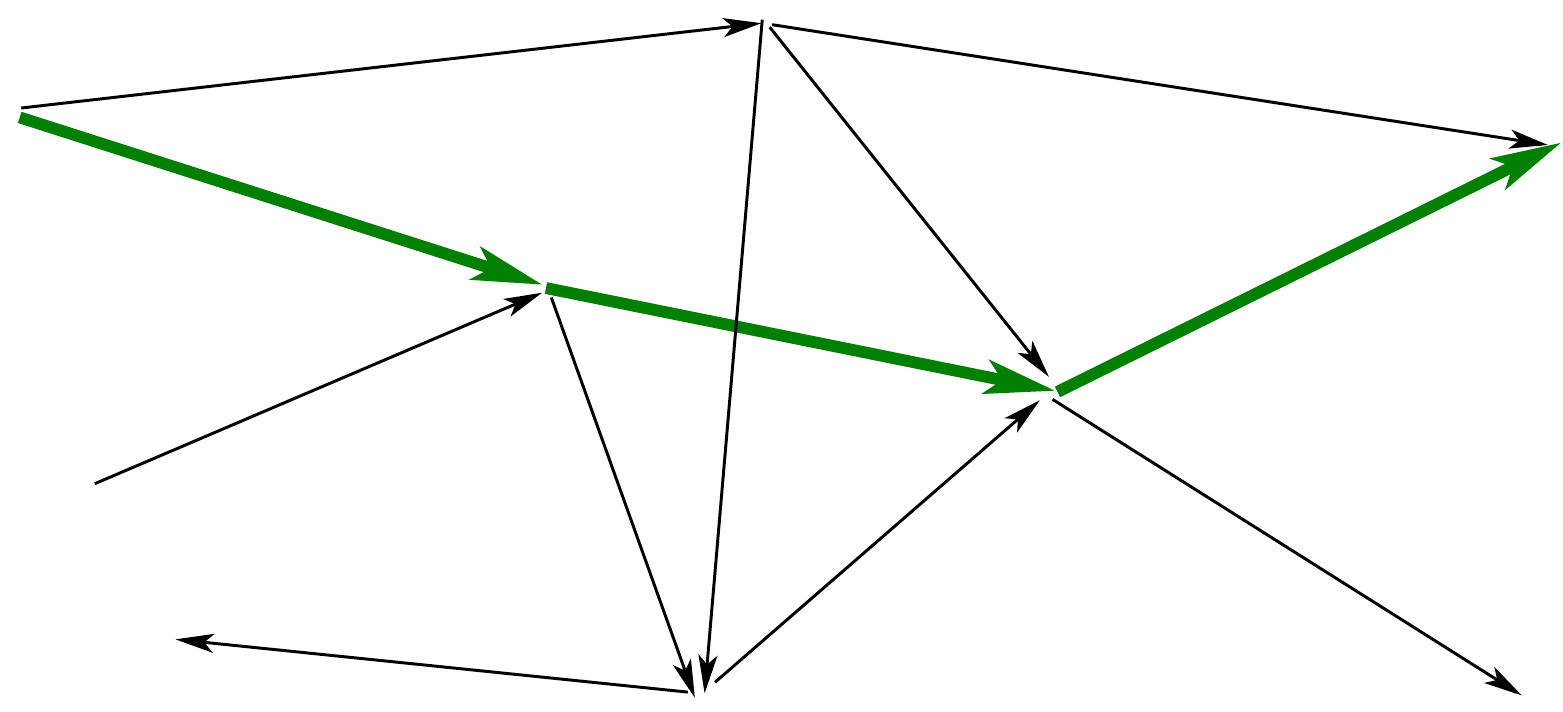}\end{overpic}
\begin{overpic}[width=0.49\textwidth]{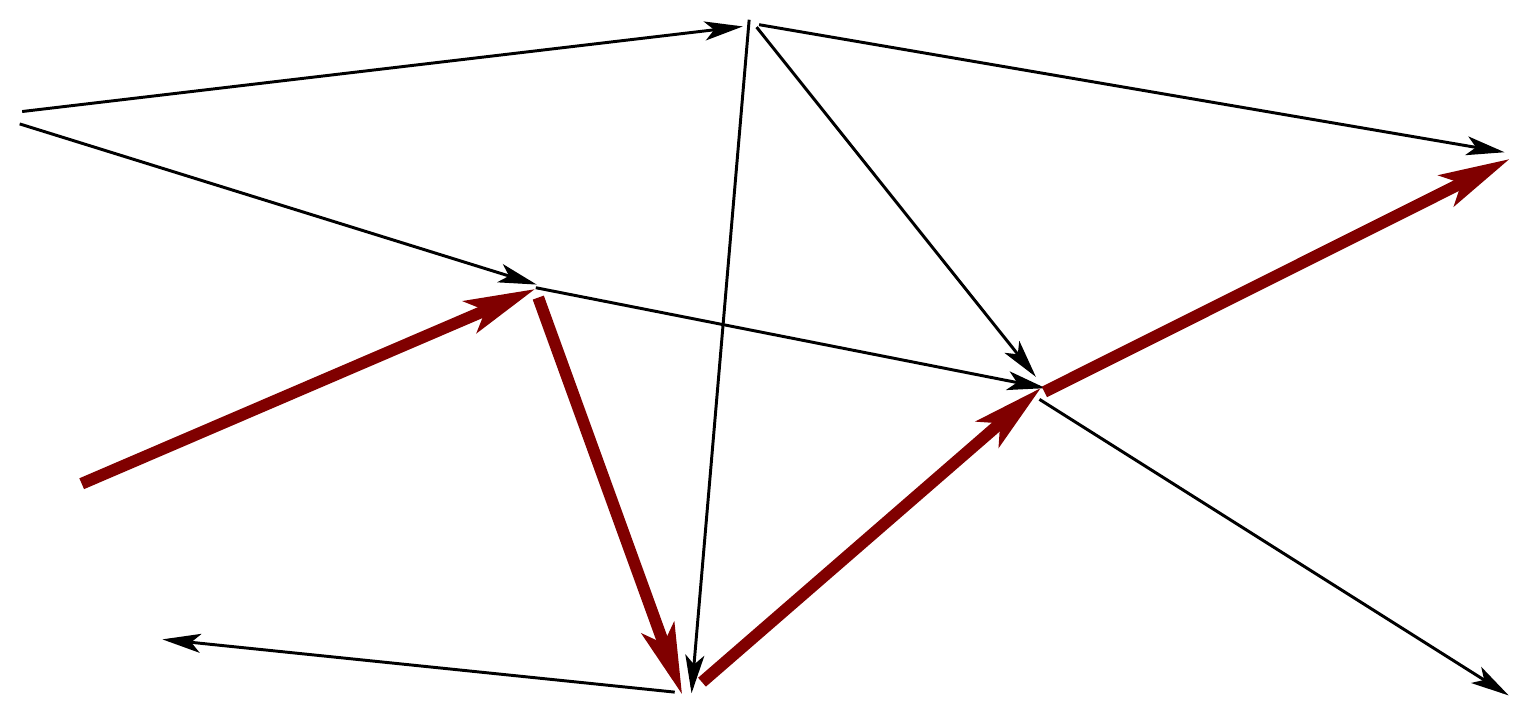}\end{overpic}} 
\caption{A generic network. Two possible paths are highlighted. Note that the two paths share one arc of the network.}
\label{fig:paths}
\end{figure}
Let us denote by $\Np$ the total number of paths on $\mathcal N$. Rather than tracking the total density $\rho$, one defines sub-densities $\mu^p$, $p=1\ldots,\Np$, each of which represents the amount of vehicles following the $p$-path. 
Clearly the dynamics of vehicles following the $p$-path are influenced by vehicles following other paths whenever the two groups of drivers share the same portion of the network. 
From the mathematical point of view, the problem is formulated as a system of $\Np$ conservation laws with discontinuous flux and no special treatment of the junction is needed (each path is seen as an uninterrupted road) \cite{bretti2014DCDS-S,briani2014NHM}. The main drawback of this approach is that the number of paths grows exponentially with the size of the network and the system becomes rapidly unmanageable.
In \cite{bretti2014DCDS-S,briani2014NHM} it was also suggested a hybrid approach which avoids to store all paths, but it was not detailed.
In the following we fill the gap giving the precise algorithm and testing it on real-size networks.

Define as usual the unknown density at grid nodes as
\begin{equation}\label{def_sol_num}
\rho_{\edge,j}^n:=\frac{1}{\Dx}\int_{x_{\edge,j-\frac12}}^{x_{\edge,j+\frac12}} \rho(y,n\Dt)dy,\qquad \edge\in\mathcal E,\quad j=1,\ldots,\Nce, \quad n=1,\ldots,\NT.
\end{equation}
Starting from the initial condition 
$\rho_{\edge,j}^0:=\frac{1}{\Dx}\int_{x_{\edge,j-\frac12}}^{x_{\edge,j+\frac12}} \rho^0(y)dy$, 
the approximate solution at any internal cell $j=2,\ldots,\Nce$--$1$ of any edge $\edge\in\mathcal E$ is easily found by the standard Godunov scheme
\begin{equation}\label{GODcelleinterne}
\rho_{\edge,j}^{n+1}=\rho_{\edge,j}^n-\frac{\Dt}{\Dx}\Big(G(\rho_{\edge,j}^n,\rho_{\edge,j+1}^n)-G(\rho_{\edge,j-1}^n,\rho_{\edge,j}^n)\Big),\quad n=0,\ldots,\NT-1,
\end{equation}
where the numerical flux $G$ is defined as
\begin{equation}\label{GodunovFlux}
G(\rho_-,\rho_+):=
\left\{
\begin{array}{ll}
\min\{f(\rho_-),f(\rho_+)\} & \textrm{if } \rho_-\leq \rho_+ \\
f(\rho_-) & \textrm{if } \rho_->\rho_+ \textrm{ and } \rho_-<\sigma \\
f(\sigma) & \textrm{if } \rho_->\rho_+ \textrm{ and } \rho_- \geq \sigma \geq \rho_+ \\
f(\rho_+) & \textrm{if } \rho_->\rho_+ \textrm{ and } \rho_+>\sigma
\end{array}
\right. .
\end{equation}

Around vertexes we proceed as follows: let us focus on a generic vertex $\vertex\in\mathcal V$ and denote by $\Nvinc$ and $\Nvout$ the number of incoming and outgoing edges at $\vertex$, respectively. 
As in the classical approaches \cite{piccolibook}, we assume that it is given a \textit{traffic distribution matrix} $A^\vertex=(\alpha^\vertex_{\textsc{r}\textsc{r}^\prime})$, $\textsc{r}=1,\ldots,\Nvinc,\ \textsc{r}^\prime=1,\ldots,\Nvout$ which prescribes how the traffic distributes in percentage from any incoming edge $\textsc{r}$ to any outgoing edge $\textsc{r}^\prime$. 
Clearly $0\leq \alpha^\vertex_{\textsc{r}\textsc{r}^\prime}\leq 1\ \forall \textsc{r},\textsc{r}^\prime$ and $\sum_{\textsc{r}^\prime}\alpha^\vertex_{\textsc{r}\textsc{r}^\prime}=1\ \forall \textsc{r}$.

Now, following the path-based approach \cite{bretti2014DCDS-S,briani2014NHM}, we look at all possible paths across the vertex $\vertex$. Having $\Nvinc$ incoming edges and $\Nvout$ outgoing edges, we have $\Nvp:=\Nvinc\times\Nvout$ possible paths.
We denote by $\mu_{\edge,j}^n(p,\vertex)$ the density of the vehicles in the cell $j$ of edge $\edge$ at time $t^n$ moving along path $p$ across vertex $\vertex$. 

Let us now focus on a generic path $p=(\edge,\edge^\prime)$ which joins edge $\einc$ with $\eout$, see Fig.\ \ref{fig:locpath}. 
\begin{figure}[h!]
\centerline{
\begin{overpic}[width=0.8\textwidth]{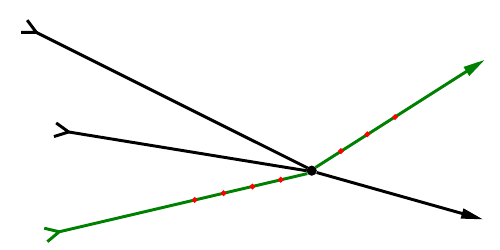}
\put(61.5,18){$\vertex$}
\put(64,19){\footnotesize \textcolor{red}{$1$}}
\put(69,22){\footnotesize \textcolor{red}{$2$}}
\put(58,12){\footnotesize \textcolor{red}{$\Nc_\edge$}}
\put(50.8,10){\footnotesize \textcolor{red}{$\Nc_\edge$--$1$}}
\put(30,10){$\edge$}
\put(80,30){$\edge^\prime$}
\put(94,32){$p=(\edge,\edge^\prime)$}
\end{overpic}
} 
\caption{Zoom around vertex $\vertex$. We show the path $p=(\edge,\edge^\prime)$ and cells' labels on that path.}
\label{fig:locpath}
\end{figure}
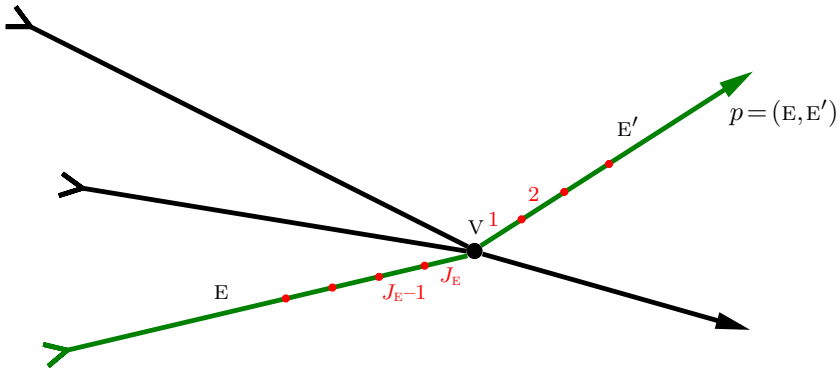
The problem is ready to be solved by the Godunov-based multi-path scheme \cite{bretti2014DCDS-S,briani2014NHM} with minor modifications for all paths of all vertexes. Dropping the indexes $p$ and $\vertex$ for readability, we have in the last cell of the incoming edge:
\begin{equation}\label{MPloc_schema_before}
\mu_{\edge,\Nce}^{n+1}=\mu_{\edge,\Nce}^n-
\frac{\Dt}{\Dx}\Bigg( \frac{\mu_{\edge,\Nce}^n}{\rho_{\edge,\Nce}^n} 
G\big(\rho_{\edge,\Nce}^n,\rho_{\edge^\prime,1}^n\big)- 
\alpha^\vertex_{\edge\edge^\prime} 
G\big(\rho_{\edge,\Nce-1}^n,\rho_{\edge,\Nce}^n\big)\Bigg),
\end{equation}
for $n=0,\ldots,\NT-1$. Note the presence of the parameter $\alpha^\vertex_{\edge\edge^\prime}$ in front of the incoming flux which tells that only a percentage of the total mass is following path $p$.

In the first cell of the outgoing edge we have instead:
\begin{equation}\label{MPloc_schema_after}
\mu_{\edge^\prime,1}^{n+1}=\mu_{\edge^\prime,1}^n-
\frac{\Dt}{\Dx}\Bigg( \frac{\mu_{\edge^\prime,1}^n}{\rho_{\edge^\prime,1}^n} 
G\big(\rho_{\edge^\prime,1}^n,\rho_{\edge^\prime,2}^n\big)- 
\frac{\mu_{\edge,\Nce}^n}{\rho_{\edge,\Nce}^n} 
G\big(\rho_{\edge,\Nce}^n,\rho_{\edge^\prime,1}^n\big)\Bigg),
\end{equation}
for $n=0,\ldots,\NT-1$. 

The algorithm is completed by summing, at any time step, the sub-densities $\mu$'s (where defined) to compute the total density $\rho$, to be used at the next time step in \eqref{GODcelleinterne}, \eqref{MPloc_schema_before}, and \eqref{MPloc_schema_after}. More precisely, we have
\begin{equation}\label{couplingmultipath}
\rho_{\edge,\Nce}^{n+1}=\sum_{q=1}^{\Nvp}\mu_{\edge,\Nce}^{n+1}(q,\vertex)\quad\text{ and }\quad
\rho_{\edge^\prime,1}^{n+1}=\sum_{q=1}^{\Nvp}\mu_{\edge^\prime,1}^{n+1}(q,\vertex).
\end{equation}
Note that \eqref{MPloc_schema_before} and \eqref{MPloc_schema_after} are \emph{systems} of $\Nvp$ equations, coupled via \eqref{couplingmultipath}, that takes into account vehicles moving along paths other than $p=(\edge,\edge')$. 
The other total densities $\rho$ appearing in \eqref{MPloc_schema_before} and \eqref{MPloc_schema_after} are instead given by \eqref{GODcelleinterne}. \REV{For the sake of clarity and dissipate any doubt, in the appendix we write explicitly the scheme in the case of a junction with two incoming and two outgoing roads.}

\subsection*{Initial condition} The scheme \eqref{GODcelleinterne}-\eqref{MPloc_schema_before}-\eqref{MPloc_schema_after} is initialized at the first time step with the total density $\rho^0$ in the cells far from the junctions, and with the sub-densities $\mu^0$'s in the cells adjacent to the junctions. If initial sub-densities are not available, we can recover them by means of the total density and the distribution matrix. 
For example, at the last cell of the incoming edge we can set
\begin{equation}\label{MPloc_DI_1}
\mu_{\edge,\Nce}^0(p,\vertex):=\alpha^\vertex_{\edge\edge^\prime}\ \rho_{\edge,\Nce}^0.
\end{equation}
Instead, at the first cell of the outgoing edge we can split the total density in equal parts, simply assuming that vehicles come equally from all incoming edges,
\begin{equation}\label{MPloc_DI_2}
\mu_{\edge^\prime,1}^0(p,\vertex):=\frac{\rho_{\edge^\prime,1}^0}{\Nvinc}.
\end{equation}

Note that the initial condition for the sub-densities $\mu$'s is needed just for initiating the numerical procedure. After that, it becomes rapidly noninfluential since scheme \eqref{GODcelleinterne} for internal cells is coupled with systems \eqref{MPloc_schema_before}-\eqref{MPloc_schema_after} by means of total density $\rho$ only (Eq.\ \ref{couplingmultipath}). Therefore, the initial subdivision is lost.


\section{Computation of the Wasserstein distance.}\label{sec:wassersteindistance}
It is well known that several problems in traffic modeling require the comparison of two density functions representing traffic conditions. In the following we list some of them:
\begin{itemize}
\item theoretical study of properties of the solution of scalar conservation laws;
\item convergence of the numerical schemes (verifying that the numerical solution is close to the exact solution);
\item calibration (finding the values of the parameters for the predicted outputs to be as close as possible to the observed ones);
\item validation (checking if the outputs are close to the observed ones);
\item sensitivity analysis (quantifying how the uncertainty in the outputs can be apportioned to different sources of uncertainty in the inputs and/or model's parameters). This is the problem we consider in this paper.
\end{itemize}

In this section we introduce the Wasserstein distance\footnote{The Wasserstein distance was first introduced by Kantorovich in 1942 and then rediscovered many times. Nowadays, it is also known as $Lip^\prime$-norm, earth mover's distance, $\bar d$-metric, Mallows distance. An important characterization is also given by the Kantorovich--Rubinstein duality theorem.} as the main ingredient for the quantification of the distance between two outputs of the LWR model (e.g., density, velocity, flux). 

\subsection{Motivations and definitions.}\label{sec:motivations}
In many papers the quantification of the closeness of two outputs is performed by means of the $L^1$, $L^2$ or $L^\infty$ distance (in space at final time or in both space and time). Although this can be satisfactory for nearly equal outputs or for convergence results, it appears inadequate for measuring the distance of largely different outputs. To see this, let us focus on the density of vehicles. In Fig.\ \ref{fig:L1vsW}
\begin{figure}[h!]
\centerline{
\begin{overpic}
[width=0.55\textwidth]{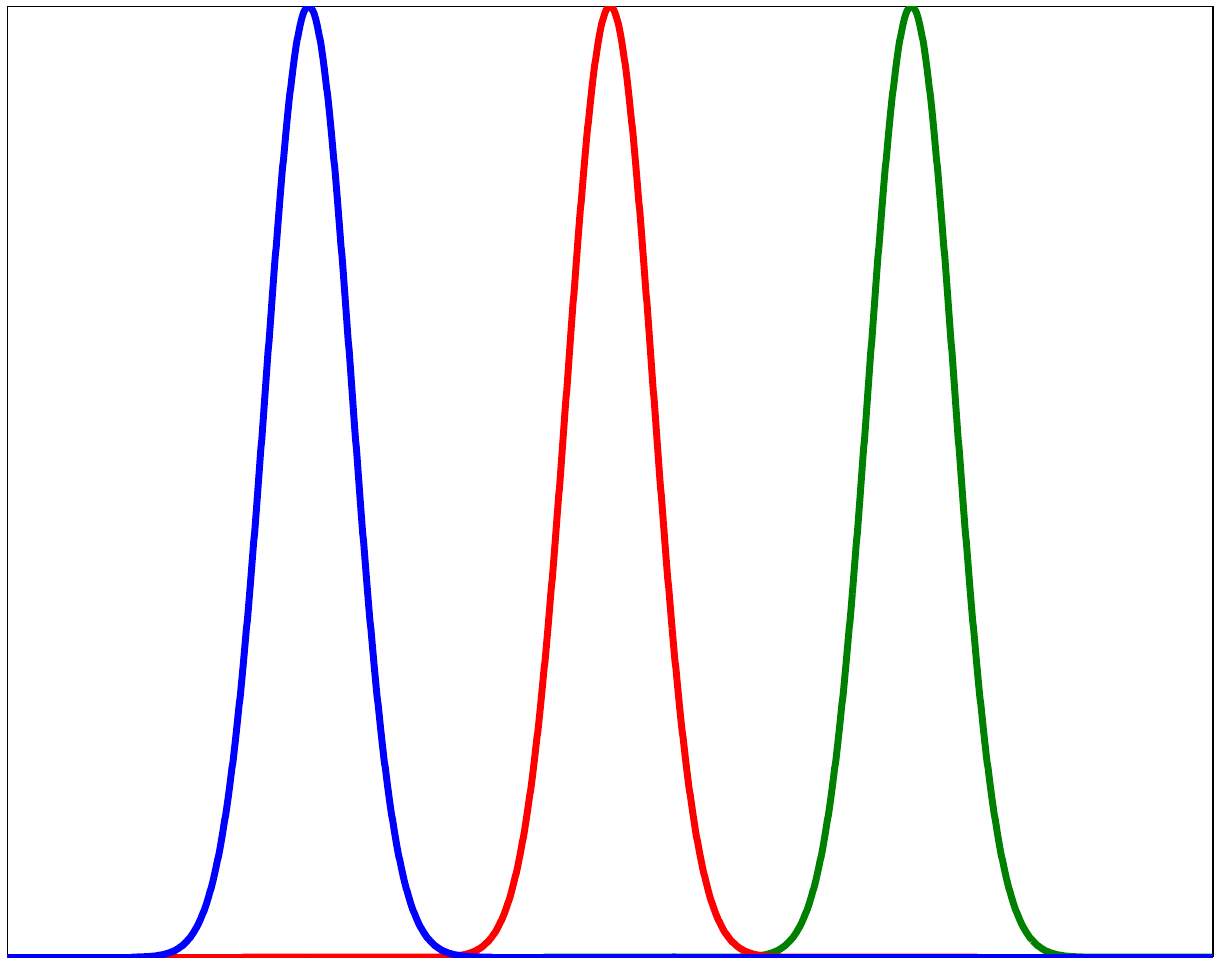}
\put(17,60){$\rho^1$}\put(41.5,60){$\rho^2$}\put(66,60){$\rho^3$}
\end{overpic}} 
\caption{Three density functions with disjoint supports.}
\label{fig:L1vsW}
\end{figure}
three density functions $\rho^i$, $i=1,2,3$, corresponding to the same total mass, say $M$, are plotted. It is plain that the $L^1$ distances between $\rho_1$ and $\rho_2$ and between $\rho_1$ and $\rho_3$ are both equal to $2M$. Similarly, all $L^p$ distances are blind with respect to variation of the densities once the supports of them are disjoint. Our perception of distance suggests instead that $\|\rho_3-\rho_1\|>\|\rho_2-\rho_1\|$, and this is exactly what Wasserstein distance guarantees, as we will see in the following.

\medskip

Let us denote by $(X,\mathfrak{D})$ a complete and separable metric space with distance $\mathfrak{D}$, and by $\mathcal B(X)$ a Borel $\sigma$-algebra of $(X,\mathfrak{D})$. Let us also denote by $\mathcal M^+(X)$ the set of non-negative finite Radon measures on $(X,\mathcal B(X))$. Let $\mis{\supply}$ (\supply \ standing for \emph{supply}) and $\mis{\demand}$ (\demand \ standing for \emph{demand}) be two Radon measures in $\mathcal M^+(X)$ such that $\mis{\supply}(X)=\mis{\demand}(X)$.
\begin{defn}[Wasserstein distance] For any $p\in[1,+\infty)$, the $L^p$-Wasserstein distance between $\mis{\supply}$ and $\mis{\demand}$ is
\begin{equation}\label{def:W}
W_p(\mis{\supply},\mis{\demand}):=\left(\inf_{\gamma\in\Gamma(\mis{\supply},\mis{\demand})}\int_{X\times X}\mathfrak{D}(x,y)^p\ d\gamma(x,y)\right)^{1/p}
\end{equation}
where
\begin{equation*}
\Gamma(\mis{\supply},\mis{\demand}):=\left\{\gamma\in\mathcal M^+(X\times X)\text{ s.t. }\gamma(A\times X)=\mis{\supply}(A),\ \gamma(X\times B)=\mis{\demand}(B),\ \forall \ A,B\subset X\right\}.
\end{equation*}
\end{defn}
Assuming that the measures $\mis{\{\supply,\demand\}}$ are absolutely continuous with respect to the Lebesgue measure, i.e.\ there are two density functions $\rho^{\{\supply,\demand\}}$ such that $d\mis{\{\supply,\demand\}}=\rho^{\{\supply,\demand\}} dx$, and considering the particular case $X=\mathbb R^n$, $\mathfrak D(x,y)=\|x-y\|_{\mathbb{R}^n}$, we have
\begin{equation}\label{def:Wbis}
W_p(\mis{\supply},\mis{\demand})=W_p(\rho^{\supply},\rho^{\demand})=\left(\inf_{T\in\mathcal T}\int_{\mathbb R^n}\|T(x)-x\|_{\mathbb{R}^n}^p\ \rho^\supply(x) dx\right)^{1/p}
\end{equation}
where
$$
\mathcal T:=\Bigg\{T:\mathbb R^n\to \mathbb R^n \text{ s.t. } \int\limits_B\rho^\demand(x)dx=\int\limits_{\{x:T(x)\in B\}}\rho^\supply(x)dx, \quad \forall B\subset \mathbb R^n \text{ bounded}\Bigg\}.
$$

Equation \eqref{def:Wbis} sheds light on the physical interpretation of the Wasserstein distance,  putting it in relation with the well known Monge--Kantorovich mass transfer problem \cite{villanibook}: a pile of, say, soil, with mass density distribution $\rho^{\supply}$, has to be moved to an excavation with mass density distribution $\rho^{\demand}$ and same total volume. Moving a unit quantity of mass has a cost which equals the distance between the source and the destination point. 
We consequently are looking for a way to rearrange the first mass onto the second which requires minimum cost.

\begin{rem}
In our framework, the mass to be moved corresponds to that of vehicles. We therefore measure the distance between two LWR solutions by computing the minimal cost to move vehicles from the scenario corresponding to one density distribution to the scenario corresponding to the other density distribution.
Moreover, we assume that the mass transfer {\em is constrained to happen along the network} $\mathcal N$ (i.e.\ $X=\mathcal N$), but the transfer {\em does not need to respect usual road laws} (road direction, traffic distribution at junctions, etc.). This is reasonable since the measure of the distance between densities is conceptually different from the physical motion of vehicles.
\end{rem}

Note that if we consider two concentrated measures $\mis{\supply}=\delta_x$ and $\mis{\demand}=\delta_y$ in $\R$, the Wasserstein distance between the two is $W_p(\delta_x,\delta_y)=|x-y|$ $\forall p$, as one would expect. This is also the desired distance quantification in the case of only two vehicles on a road, located in $x$ and $y$ respectively.

\medskip

Unfortunately, definitions \eqref{def:W} and \eqref{def:Wbis}, though elegant, are not suitable for numerical approximation. Limiting our attention to the real line, the problem is easily solved by using nice alternative definitions of Wasserstein distance, like the ones reported in the following. 

If $p=1$, we have
\begin{equation}\label{W1Da}
W_1(\rho^\supply,\rho^\demand)=\int_\mathbb{R} |F^\supply(x)-F^\demand(x)|dx,\qquad F^{\{\supply,\demand\}}(x):=\int_{-\infty}^x \rho^{\{\supply,\demand\}}(x)dx.
\end{equation}

If $p=2$, we have
\begin{equation}\label{W1Db}
W_2(\rho^\supply,\rho^\demand)=\left(\int_{\mathbb{R}}|T^*(x)-x|^2\ \rho^\supply(x) dx\right)^{1/2}
\end{equation}
where $T^*$ is easily found as the function that satisfies
$$
\int_{y<T^*(x)} \rho^\demand(y)dy=\int_{y<x} \rho^\supply(y)dy,\qquad \forall x\in\mathbb R.
$$

The case of a network is more complicated. In the following section we propose an algorithm to approximate the Wasserstein distance on a network, reformulating the problem on a discrete graph.

\subsection{Discrete formulation on graph.}\label{sec:H}
A network can be always approximated by a discrete graph at the cost of a loss of resolution. 
Resorting to the discretization introduced in Sec.\ \ref{sec:basicdefs}, we can create a new, undirected graph whose vertexes coincide with the centers of the cells, and they are connected in agreement with the network, see Fig.\ \ref{fig:networktograph}. 
\begin{figure}[h!]
\centerline{
\begin{overpic}[width=0.32\textwidth]{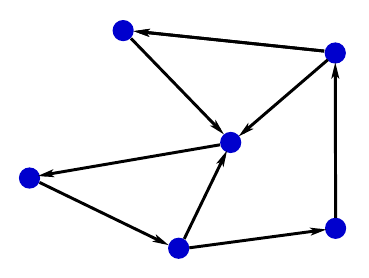}\end{overpic}
\begin{overpic}[width=0.32\textwidth]{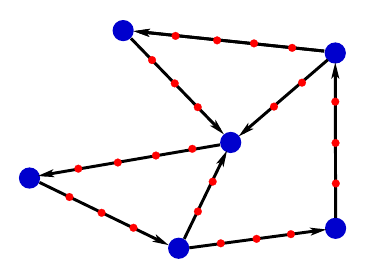}
\put(60,68){\footnotesize $C_j^\edge$}
\put(89,65){$\vertex$}
\put(50,59){$\edge$}
\put(32,35){\tiny $\Dx$}
\end{overpic}
\begin{overpic}[width=0.32\textwidth]{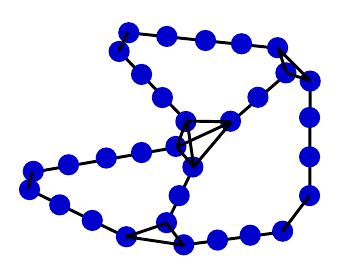}\end{overpic}} 
\caption{Original directed graph $\mathcal G$ (left), network $\mathcal N$ built on $\mathcal G$ discretized with space step $\Dx$ (center), and undirected graph $\GW$ built on the discretized network $\mathcal N$ (right).}
\label{fig:networktograph}
\end{figure}
The graph created by this procedure will be denoted hereafter by $\GW$. The number of vertices of $\GW$ equals the total number of cells in $\mathcal N$, therefore will be denoted by $\Nc$.
In order to complete the discretization procedure, \emph{the mass distributed on each cell is accumulated to the vertex located at the center of the cell}. 
Doing this, the problem is reformulated as an optimal mass transportation problem on a graph. 
The new problem clearly approximates the original one on the network $\mathcal N$ and the approximation error is controlled by $\Dx$. 
Focusing on our particular case, after the numerical approximation of the LWR model, we are left with a single value $\rho_{\edge,j}$ for each cell $C_{\edge,j}$ of the network, which represents the average density in that cell, see \eqref{def_sol_num}. This means that the numerical procedure returns a constant density $\rho(x)\equiv\rho_{\edge,j}$ for all $x\in C_{\edge,j}$, which must be accumulated in the centre $x_{\edge,j}$ of the cell.

At this point one can resort to classical problems (see Hitchcock's paper \cite{hitchcock1941JMP}) and methods (see e.g., \cite[Sec.\ 6.4.1]{santabrogiobook} and \cite[Chap.\ 19]{sinhabook}), recasting the problem in the framework of linear programming (LP).

\medskip

Let us enumerate the vertexes of $\GW$ by $j=1,\ldots,\Nc$, and denote by $\rho^\supply_j$, $\rho^\demand_j$, the supply and demand densities concentrated in vertex $j$, respectively.
Following the mass transport interpretation, the supply mass at vertex $j$ is $\supply_j:=\rho^\supply_j\Dx$, and the demand mass is $\demand_j:=\rho^\demand_j\Dx$. 
Let $c_{jk}$ be the cost of shipping a unit quantity of mass from the origin $j\in\{1,\ldots,\Nc\}$ to the destination $k\in\{1,\ldots,\Nc\}$. Here we define $c_{jk}$ as the length of the shortest path joining $j$ and $k$ on $\GW$, which can be easily found by, e.g., the Dijkstra algorithm \cite{dijkstra1950NM}. 
Let $x_{jk}$ be the (unknown) quantity shipped from the origin $j$ to the destination $k$. The problem is then formulated as
\begin{equation}\label{hitchcock}
\begin{tabular}{ll}
minimize & $\mathcal H:=\sum\limits_{j=1}^\Nc \sum\limits_{k=1}^\Nc c_{jk} x_{jk}$ \\ [3mm]
subject to & $\sum\limits_k x_{jk}=\supply_j,\quad \forall j$ \\  [2mm]
 & $\sum\limits_j x_{jk}=\demand_k,\quad \forall k$ \\  [2mm]
 & $x_{jk}\geq 0$.
\end{tabular}
\end{equation}
Note that the solution satisfies $x_{jk}\leq\min\{\supply_j,\demand_k\}$ since one cannot move more than $\supply_j$ from any source vertex $j$ and it is useless to bring more than $\demand_k$ to any sink vertex $k$.
From \eqref{hitchcock} it is easy to recover a standard LP problem
\begin{equation}\label{PL}
\begin{tabular}{ll}
minimize & $\mathbf{c}^\intercal \mathbf{x}$ \\ [3mm]
subject to & $\mathbf{A}\mathbf{x}=\mathbf{b}$ \\  [2mm]
 & $\mathbf{x}\geq 0$,
\end{tabular}
\end{equation}
simply defining
\begin{eqnarray*}
& &\mathbf{x}:=(x_{11},x_{12},\ldots,x_{1\Nc},x_{21},x_{22},\ldots,x_{2\Nc},\ldots,x_{\Nc 1},\ldots,x_{\Nc\Nc})^\intercal \\
& &\mathbf{c}:=(c_{11},c_{12},\ldots,c_{1\Nc},c_{21},c_{22},\ldots,c_{2\Nc},\ldots,c_{\Nc 1},\ldots,c_{\Nc\Nc})^\intercal \\
& &\mathbf{b}:=(\supply_1,\ldots,\supply_{\Nc},\demand_1,\ldots,\demand_{\Nc})^\intercal 
\end{eqnarray*}
and $\mathbf{A}$ as the $2\Nc\times \Nc^2$ sparse matrix
\begin{equation}
\mathbf{A}:=
\left[
\begin{array}{ccccc}
\mathbbm{1}_\Nc & 0 & 0 & \cdots & 0 \\
0 & \mathbbm{1}_\Nc & 0 & \cdots & 0 \\
0 & 0 & \mathbbm{1}_\Nc & \cdots & 0 \\
\vdots & \vdots & \vdots & \ddots & \vdots \\
0 & 0 & 0 & \cdots & \mathbbm{1}_\Nc \\
I_\Nc & I_\Nc & I_\Nc & I_\Nc & I_\Nc \\
\end{array}
\right]
\end{equation}
where $I_\Nc$ is the $\Nc\times \Nc$ identity matrix and $\mathbbm{1}_\Nc:=(\underbrace{1\ 1\ \cdots \ 1}_{\Nc \text{ times}})$.

\subsection{Error analysis.}\label{sec:erroranalysis}
Before focusing on the sensitivity analysis of the LWR model, it is useful to quantify the error introduced by the LP-based method presented above in computing the exact Wasserstein distance. 
We do that in the general case, without restricting ourselves to piecewise constant density functions.
\begin{prop}\label{prop:W-H}
Let $\rho^\supply,\rho^\demand:\mathcal N\to\R$ two densities defined on a network $\mathcal N$ such that 
$$M=\int_{\mathcal N}\rho^\supply dx=\int_{\mathcal N}\rho^\demand dx.$$ 
Then,
\begin{equation}\label{prop:W-H-tesi}
|W(\rho^\supply,\rho^\demand)-\mathcal H(\rho^\supply,\rho^\demand)|\leq M\Dx
\end{equation} 
where hereafter $W$ denotes the Wasserstein distance $W_1$ and $\mathcal H$ is the solution of the problem \ref{hitchcock}.
\end{prop}
\begin{proof}
To begin with, let us focus on a generic cell $C_j$ of the network. In accordance with the optimal flow (found \textit{a posteriori} as the solution of the optimal mass problem), the mass in $C_j$ is transferred in one or more cells of the network. Let us denote by $m_{jk}$ the mass which is moved from cell $C_j$ to $C_k$ for some $k=1,\ldots,J$ (including $k=j$). 
Let us also denote by $\tau^\supply_{jk}(\cdot)$ the density profile (with supp$(\tau^\supply_{jk})\subseteq C_j$) associated to the leaving mass $m_{jk}$ in $C_j$ and by $\tau_{jk}^\demand(\cdot)$ the density profile (with supp$(\tau_{jk}^\demand)\subseteq C_k$) associated to the arriving mass in $C_k$, see Fig.\ \ref{fig:teo_W}. 
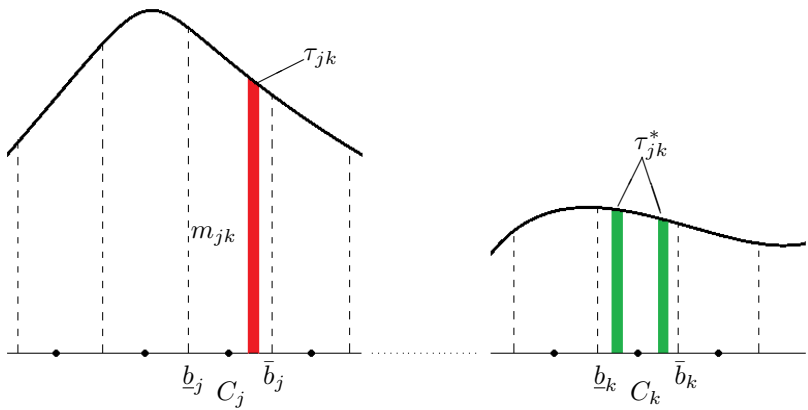
\begin{figure}[h!]
\centerline{
\begin{overpic}[width=0.8\textwidth]{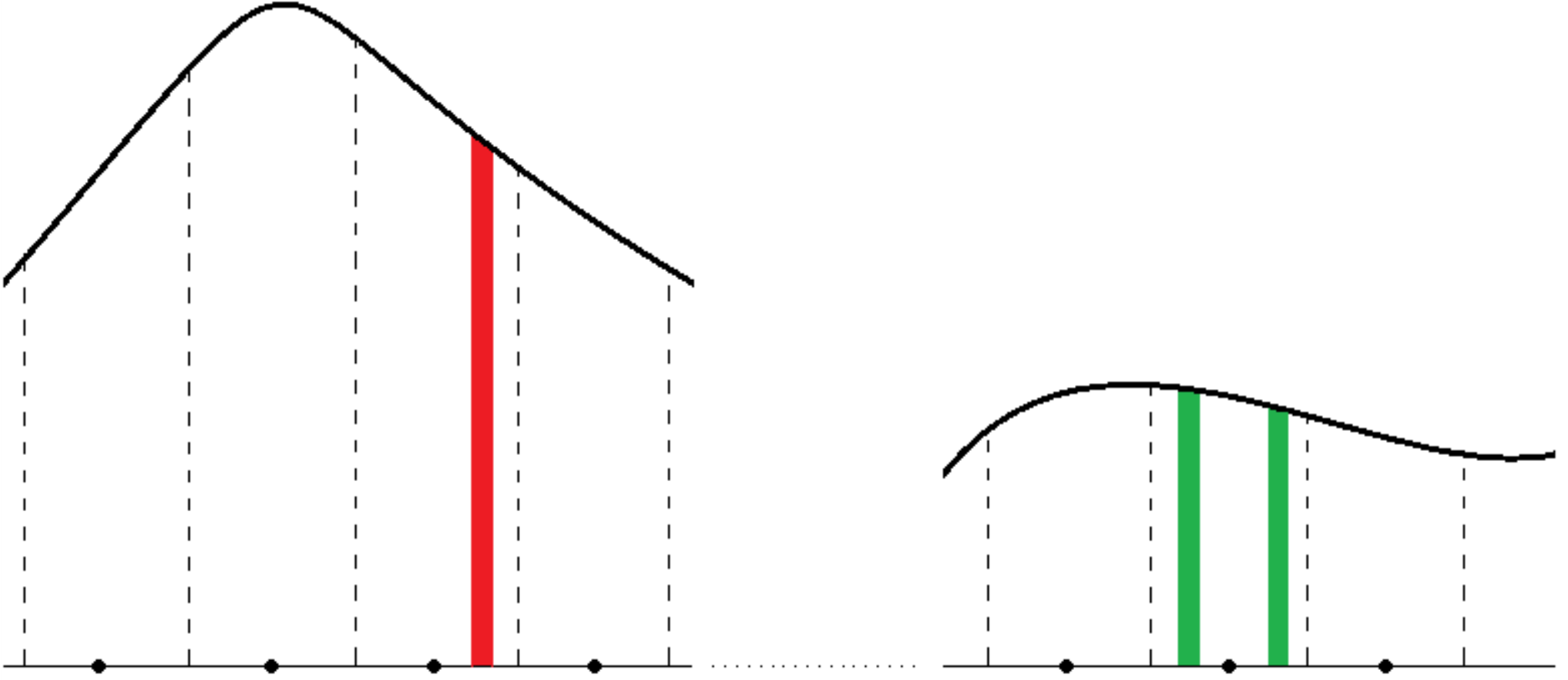}
\put(22,-3){$\underline b_j$}\put(32,-3){$\bar b_j$}\put(26.2,-5){$C_j$}
\put(73,-3){$\underline b_k$}\put(83,-3){$\bar b_k$}\put(77.5,-5){$C_k$}
\put(23,15){$m_{jk}$}\put(37,37){$\tau_{jk}$}\put(78,26){$\tau^*_{jk}$}
\put(79,25){\line(-1,-2){3}}\put(79,25){\line(1,-3){2.3}}
\put(31,34){\line(2,1){6}}
\end{overpic}} 
\vspace*{12pt}
\caption{Proposition \ref{prop:W-H}. Mass $m_{jk}$ moving from $C_j=[\underline b_j,\bar b_j)$ to $C_k=[\underline b_k,\bar b_k)$.}
\label{fig:teo_W}
\end{figure}
By definition we have
$$
m_{jk}=\int_{C_j}\tau^\supply_{jk}(x)dx=\int_{C_k}\tau_{jk}^\demand(x)dx.
$$
Let us denote by $\underline b_j:=x_{j-\frac12}$, $\bar b_j:=x_{j+\frac12}$,  and similarly by $\underline b _k$, $\bar b_k$, the two border points of the cell $j$ and $k$, respectively.

By suitably accumulating the masses at the \emph{borders} of the cells, and recalling that the discrete approach requires instead to accumulate the masses at the \emph{centers} of the cells, we have:
\\
\emph{Case A:} $j\neq k$. 
\begin{equation}\label{stima_jneqk_leq}
W(\tau^\supply_{jk},\tau_{jk}^\demand)\leq
m_{jk} 
\max_{\substack{b_j\in\{\bar b_j,\underline b _j\} \\ b_k\in\{\bar b_k,\underline b _k\}}}
W(\delta_{b_j},\delta_{b_k})=
\mathcal H(\tau^\supply_{jk},\tau_{jk}^\demand)+2\frac{m_{jk}\Dx}{2},
\end{equation}
and, equivalently,
\begin{equation}\label{stima_jneqk_geq}
W(\tau^\supply_{jk},\tau_{jk}^\demand)\geq
m_{jk} 
\min_{\substack{b_j\in\{\bar b_j,\underline b _j\} \\ b_k\in\{\bar b_k,\underline b _k\}}}
W(\delta_{b_j},\delta_{b_k})=
\mathcal H(\tau^\supply_{jk},\tau_{jk}^\demand)-2\frac{m_{jk}\Dx}{2}
\end{equation}
(where the additional distance $\pm 2\frac{m_{jk}\Dx}{2}$ comes from moving the mass from the borders to the centers of the cells in $C_j$ and $C_k$).
\\
\emph{Case B:} $j=k$. 
\begin{equation}\label{stima_j=k_leq}
\mathcal H(\tau^\supply_{jk},\tau_{jk}^\demand)=0
\qquad\text{and}\qquad
0\leq W(\tau^\supply_{jk},\tau_{jk}^\demand)\leq m_{jk}\Dx,
\end{equation}
then we still have
$$
W(\tau^\supply_{jk},\tau_{jk}^\demand)\leq
\mathcal H(\tau^\supply_{jk},\tau_{jk}^\demand)+m_{jk}\Dx
\qquad\text{and}\qquad
W(\tau^\supply_{jk},\tau_{jk}^\demand)\geq
\mathcal H(\tau^\supply_{jk},\tau_{jk}^\demand)-m_{jk}\Dx.
$$
as in \eqref{stima_jneqk_leq}-\eqref{stima_jneqk_geq}.

Summing up we obtain, by \eqref{hitchcock},
$$
W(\rho^\supply,\rho^\demand)=\sum_j\sum_k W(\tau^\supply_{jk},\tau_{jk}^\demand)\leq
\sum_j\sum_k [\mathcal H(\tau^\supply_{jk},\tau_{jk}^\demand)+m_{jk}\Dx]=
\mathcal H(\rho^\supply,\rho^\demand)+M\Dx
$$
and
$$
W(\rho^\supply,\rho^\demand)=\sum_j\sum_k W(\tau^\supply_{jk},\tau_{jk}^\demand)\geq
\sum_j\sum_k [\mathcal H(\tau^\supply_{jk},\tau_{jk}^\demand)-m_{jk}\Dx]=
\mathcal H(\rho^\supply,\rho^\demand)-M\Dx.
$$
Finally we have
$$
|W(\rho^\supply,\rho^\demand)-\mathcal H(\rho^\supply,\rho^\demand)|\leq M\Dx.
$$
It is also easy to prove that this estimate is actually sharp. To see this it is sufficient to consider the one-dimensional case $\mathcal N=\R$ and choose $\rho^{\supply}=M\delta_{\bar{b}_j}$ and $\rho^{\demand}=M\delta_{\underline{b}_{j+1}}$.
\end{proof}
\begin{figure}[b!]
\centerline{
\begin{overpic}[width=0.45\textwidth]{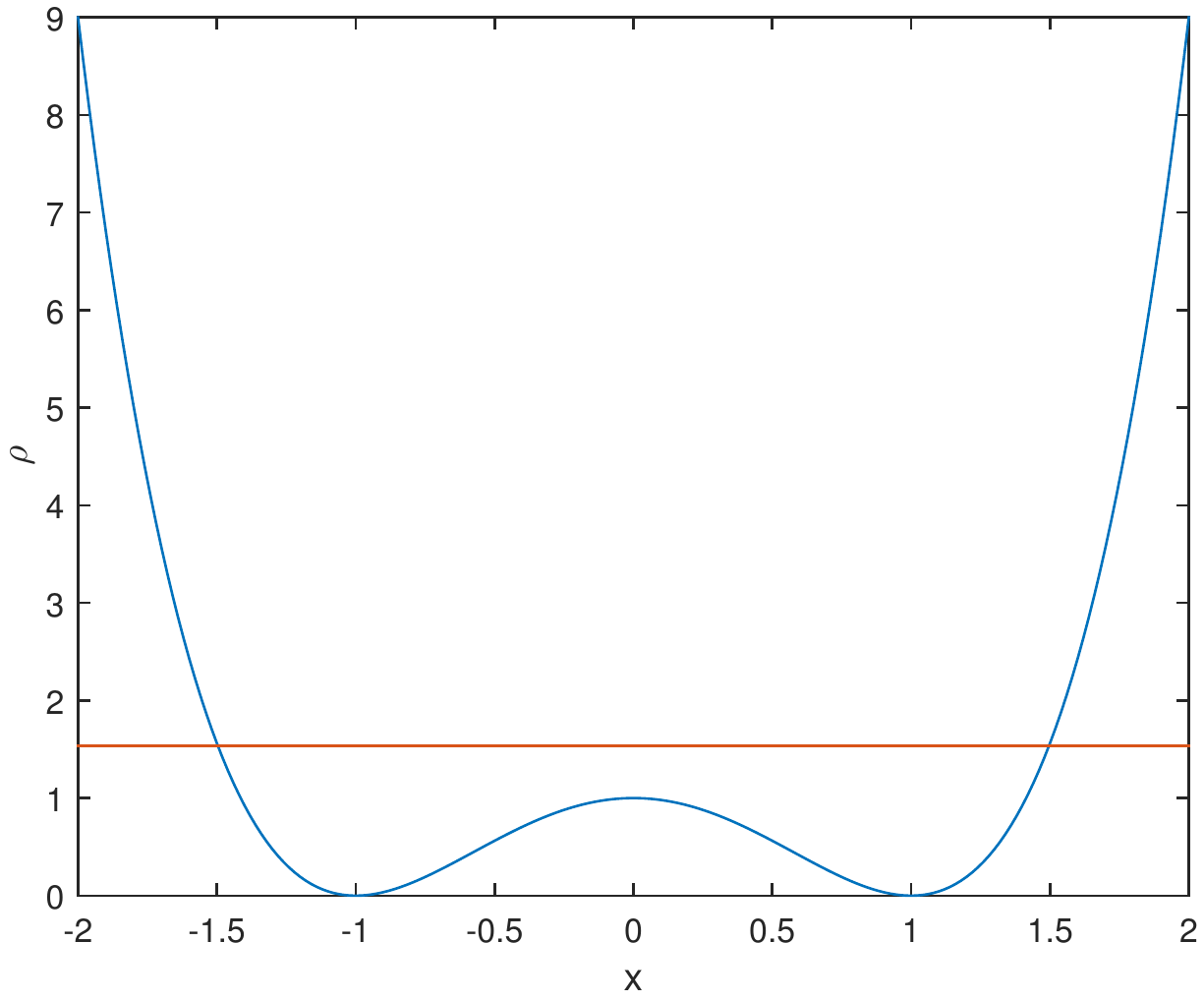}
\put(88,50){$\rho^\supply$}\put(33,24){$\rho^\demand$}
\end{overpic}\quad\qquad
\begin{overpic}[width=0.47\textwidth]{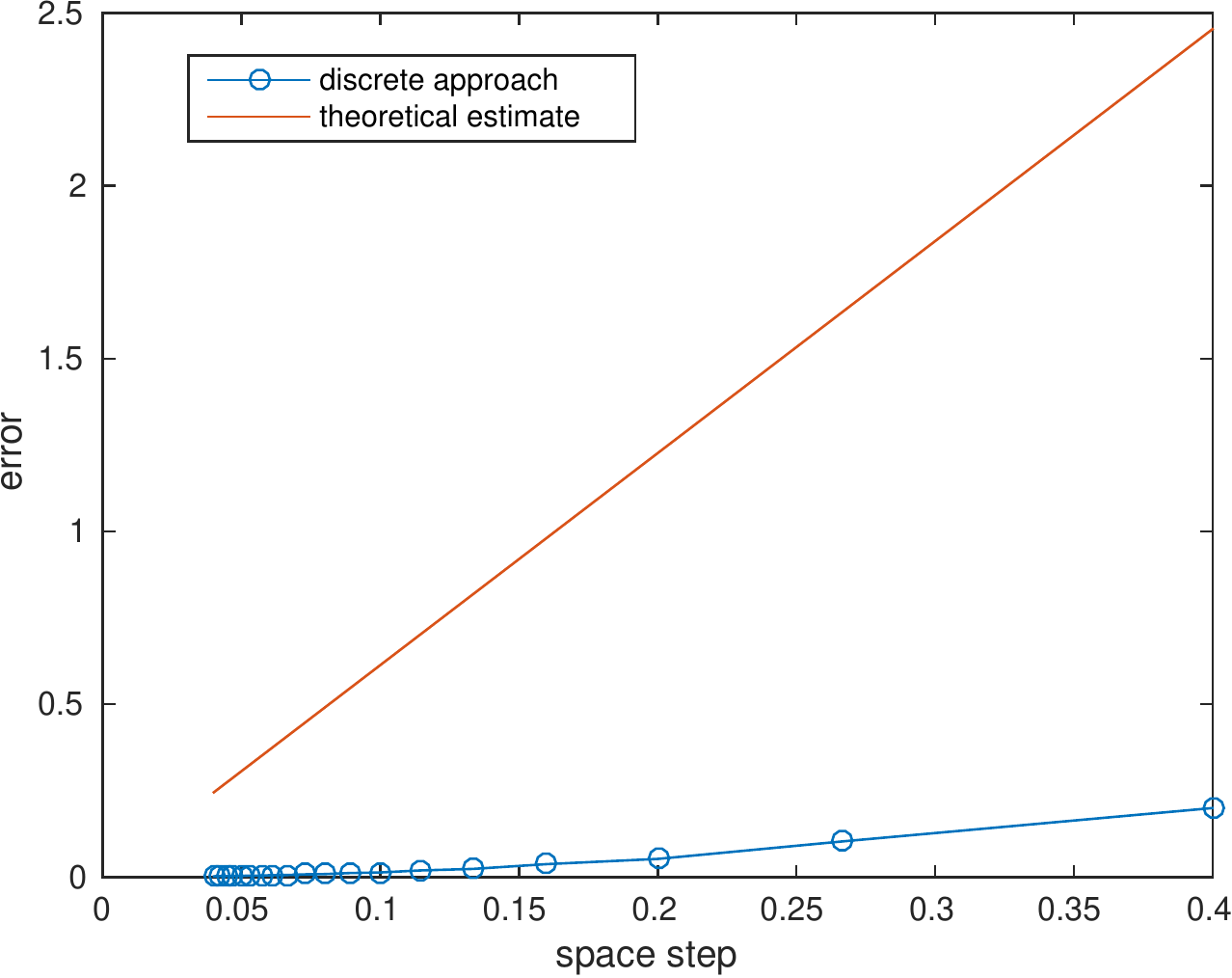}
\end{overpic}} 
\caption{Exact vs.\ approximate Wasserstein distance. Functions $\rho^\supply$, $\rho^\demand$ (left) and convergence of $|W-\mathcal H|$ as $\Dx\to 0$ (right).}
\label{fig:test0}
\end{figure}

In the following we test the discrete approach described above against a one-dimensional problem where the Wasserstein distance can be analytically computed. 
We define
$$
\rho^\supply(x)=\left\{
\begin{array}{ll}
x^4-2x^2+1, & \ x\in[-2,2]\\
0, & \ \text{otherwise}
\end{array}
\right. \qquad\qquad\text{and}\qquad\qquad
\rho^\demand(x)\equiv\frac{23}{15},
$$
see Fig.\ \ref{fig:test0}(left). 
Note that the total mass is equal, i.e.\ $M=\int_{\R}\rho^\supply=\int_{\R}\rho^\demand=\frac{92}{15}$. The exact Wasserstein distance between the two densities can be easily computed by using \eqref{W1Da}, obtaining $W(\rho^\supply,\rho^\demand)=3.2$. In Fig.\ \ref{fig:test0}(right) we report the value of the error $|W-\mathcal H|$ as a function of the space step $\Dx$ used to discretize the interval $[-2,2]$, and we compare it with the theoretical estimate given by Prop.\ \ref{prop:W-H}.
We note that in this special case the measured convergence rate is superlinear and the error is much lower than the theoretical estimate.


\section{Sensitivity analysis.}\label{sec:sensitivity}
In this section we employ the discrete approach described in Sec.\ \ref{sec:H} to perform a sensitivity analysis of the LWR model. 
To solve the LP problem we used the GLPK\footnote{https://www.gnu.org/software/glpk/} free C library.

\medskip

For numerical tests we consider the ``Manhattan''-like two-way road network depicted in Fig.\ \ref{fig:manhattan_vuota}. This choice is motivated by the fact that it allows one to easily compare networks of different size. Given the number $\Nnl$ of junctions per side, we get $4\Nnl(\Nnl-1)$ roads and $\Nnl^2$ junctions. Roads are numbered starting from those going rightward, then leftward, upward, and finally downward. The length of each road is $L_\edge=1$ and, if not otherwise stated, $\Dx=0.1$ ($\Nce=10$, $J=40\Nnl(\Nnl-1)$).
\begin{figure}[h!]
\centerline{
\begin{overpic}[width=0.85\textwidth]{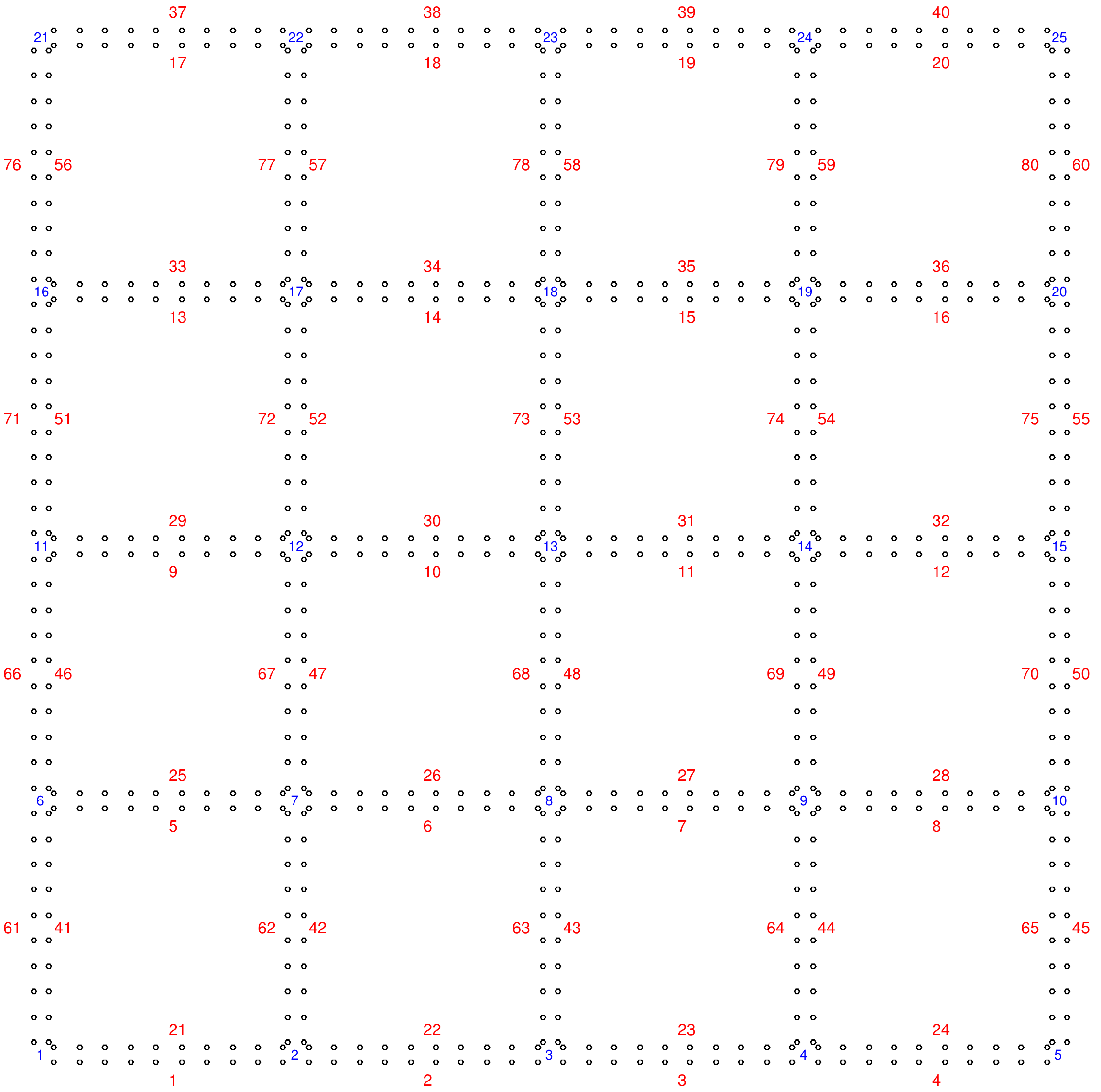}
\put(17,0.5){$\rightarrow$}\put(11.5,4.7){$\leftarrow$}
\put(5,17.5){$\uparrow$}\put(0.4,11){$\downarrow$}
\end{overpic}
}
\caption{Manhattan-like road network with $\Nnl=5$ and $J_\edge=10$. We draw the centers of the cells and report the numbering of roads and junctions. Roads are actually two-way, the small gap between lanes going in opposite directions is left for visualization purpose only. Road directions are indicated by the arrows at the bottom-left corner.}
\label{fig:manhattan_vuota}
\end{figure}
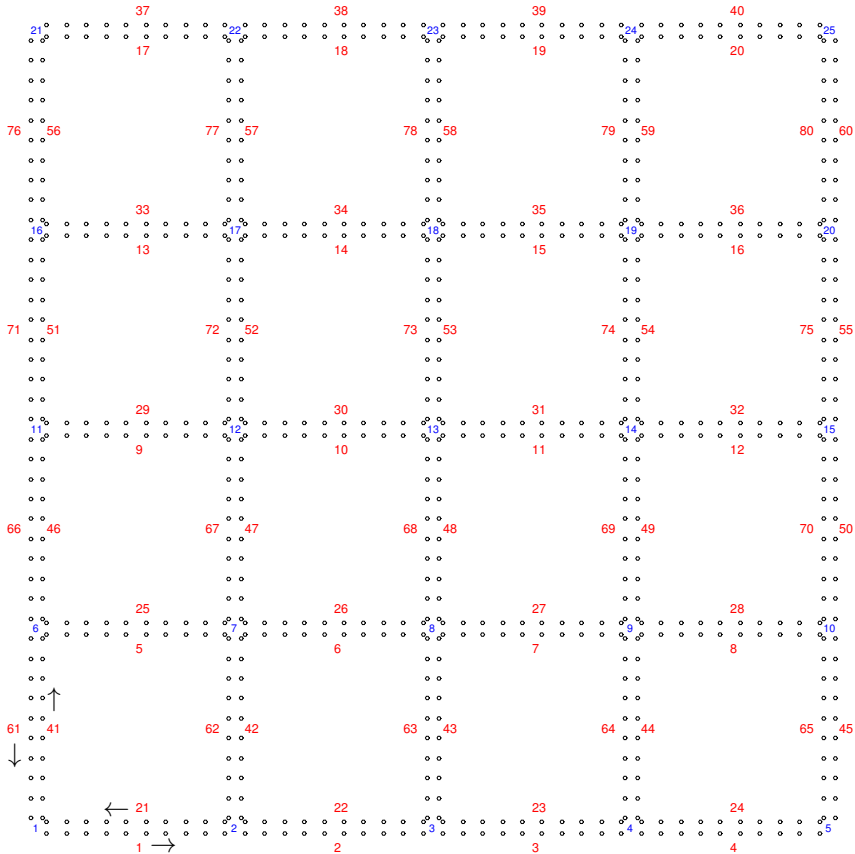

In order to fairly compare simulations with different number of vehicles, we report the \textit{normalized} approximate Wasserstein distance 
\begin{equation}
\Hnorm:=\frac{\sum_j\sum_k c_{jk} x_{jk}^*}{M},
\end{equation}
where $M=\sum_j \supply_j=\sum_j \demand_j$, and $x_{jk}^*$ is the solution of the LP problem \eqref{PL}.
%
%
%
%
%
%
%
%
%
%

\subsection{Sensitivity to initial data.}\label{sec:sens.ID.R1}
In this test we measure the sensitivity to the initial position of vehicles. 
The goal is to quantify the impact of a possible error in locating vehicles at initial time (but still catching the correct amount of vehicles). In addition, this preliminary test aims at investigating some conceptual and numerical aspects of the proposed procedure. In particular we show the difference between Wasserstein and $L^1$ distance (see Section \ref{sec:motivations}) and we study the convergence $\Hnorm \to W$ as $\Dx\to 0$ (see Section \ref{sec:erroranalysis}). 

The parameters which remain fixed in this test are
\begin{itemize}
\item \emph{Fundamental diagram}: $\sigma=0.3$ and $\fmax=0.25$ (see \eqref{FD}).
\item \emph{Distribution matrix}: 
$$
\alpha^\vertex_{\textsc{r} \textsc{r}^\prime}=\frac{1}{\Nvout},\qquad \forall \vertex\in\mathcal V,\quad \textsc{r}=1,\ldots,\Nvinc, \quad \textsc{r}^\prime=1,\ldots,\Nvout.
$$
\end{itemize}
We consider the following two initial conditions, see Fig.\ \ref{fig:test1_ID}:
for all $\edge\in\mathcal{E}$ and $j=1,\ldots,\Nce/2$,
\begin{equation}\label{test1_ID}
\begin{array}{lr}
\rho^{\supply,0}_{\edge,j} = \left\{\begin{array}{ll}
0.5, &\mbox{ on rightward roads}\\
\smallskip
0, & \mbox{ elsewhere},
\end{array}\right.
&\quad\quad
\rho^{\demand,0}_{\edge,j} = \left\{\begin{array}{ll}
0.5, &\mbox{ on leftward roads}\\
\smallskip
0, & \mbox{ elsewhere}.
\end{array}\right.
\end{array}
\end{equation}
\begin{figure}[h!]
\centerline{
\includegraphics[width=0.447\textwidth]{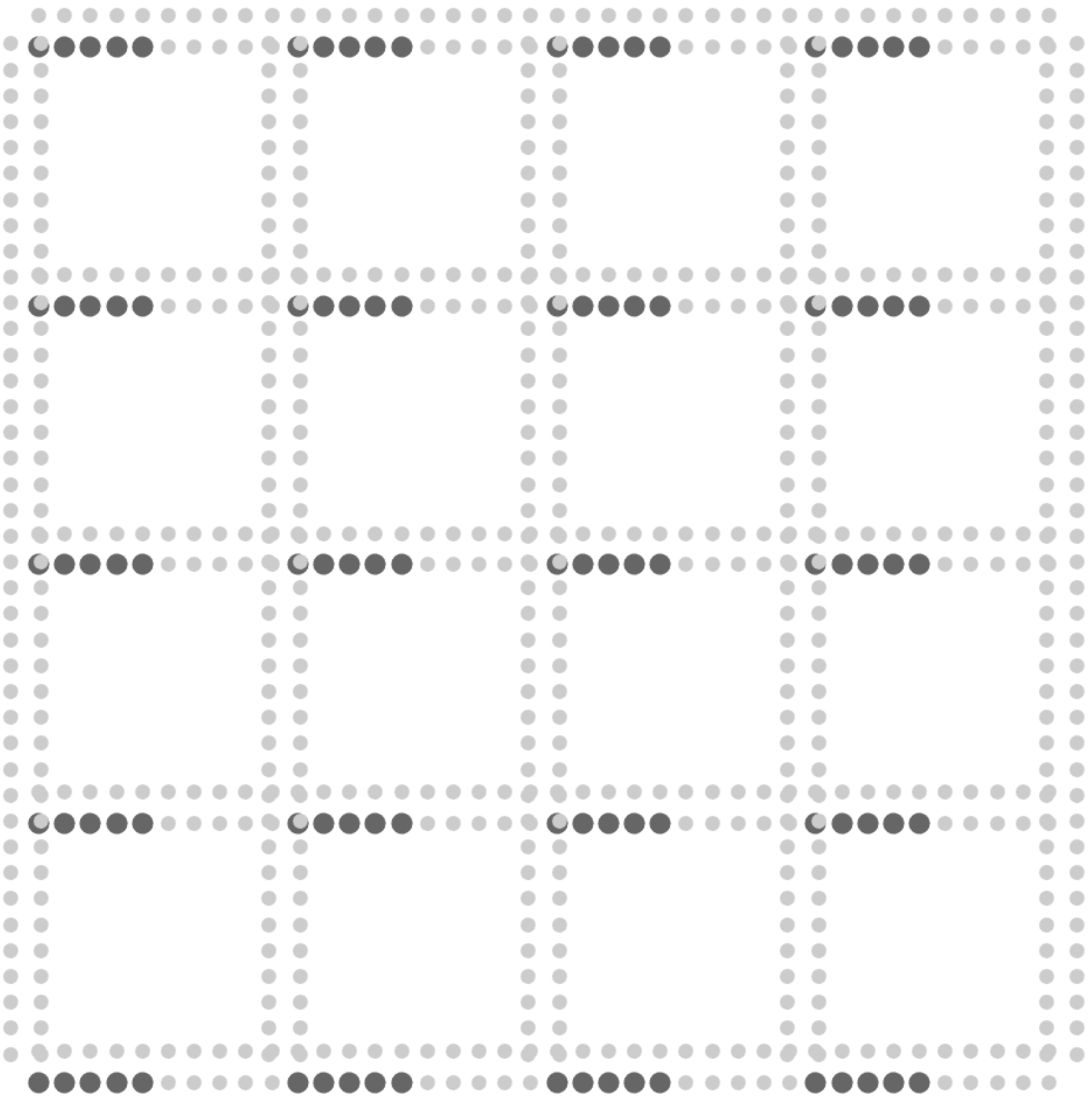}\qquad
\includegraphics[width=0.49\textwidth]{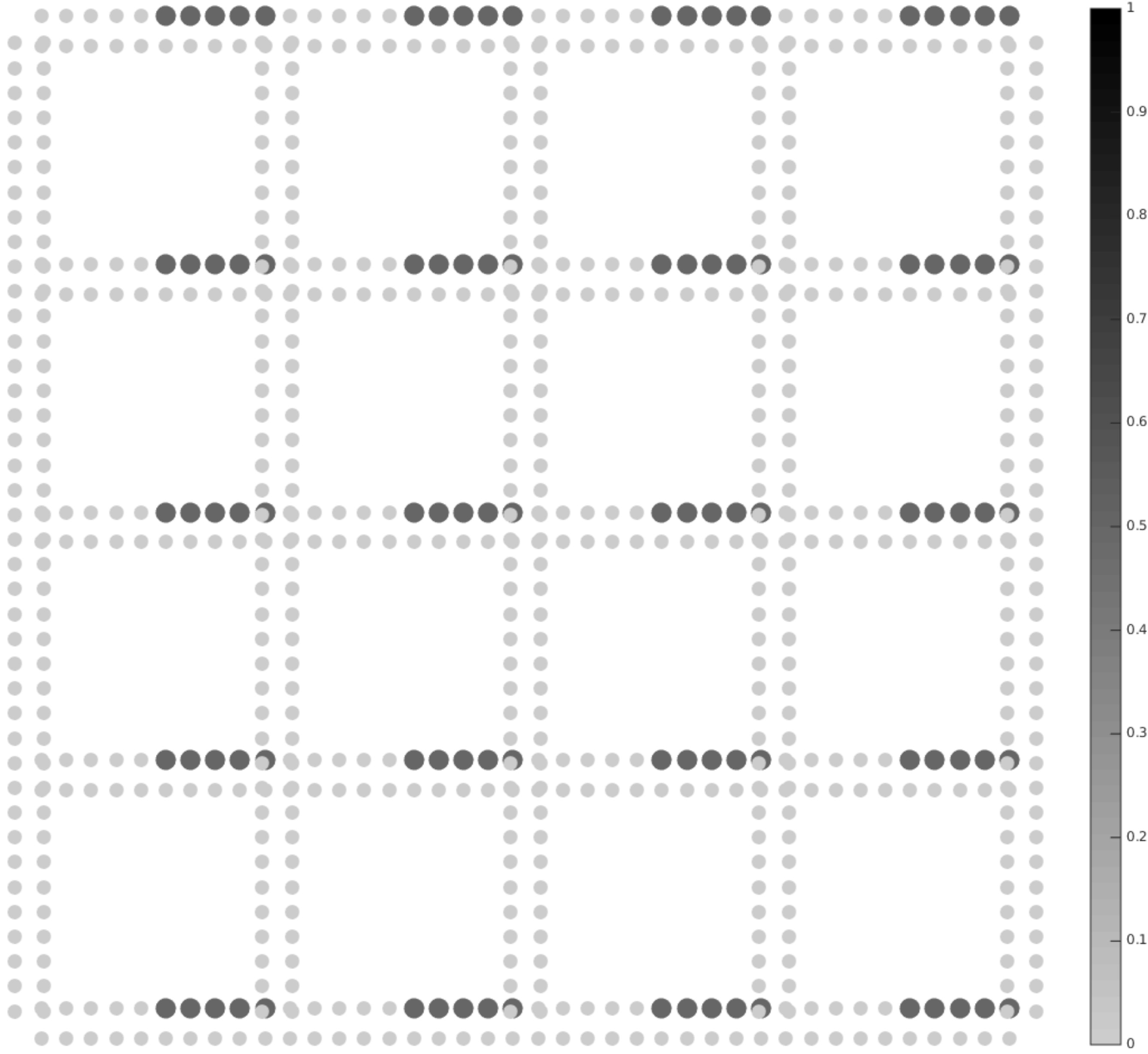}
}
\caption{Sensitivity to initial data. $\rho^{\supply,0}$ (left) and $\rho^{\demand,0}$ (right).} 
\label{fig:test1_ID}
\end{figure}
\begin{rem}\label{rem:tuttoazero}
Due to the uniform traffic distribution at junctions, the density tends to become constant on the whole network as $t\to +\infty$, regardless of the initial datum. As a consequence, we expect that the distance  between $\rho^\supply(t)$ and $\rho^\demand(t)$ (no matter how defined) tends to 0 as $t\to +\infty$. 
\end{rem}

\subsubsection{Comparison with $L^1$ distance.}\label{sec:L1}
In this test we compare the approximate Wasserstein distance with the discrete $L^1$ distance (normalized with respect to the mass as well), here denoted by $\hat{\mathcal L}^1$ and defined by
\begin{equation}\label{l1_error}
\hat{\mathcal L}^1(\rho^\supply(\cdot,t),\rho^\demand(\cdot,t)):=\frac{\Dx}{M}\sum_{\edge\in\mathcal{E}}\sum_{j=1}^{\Nce}|\rho^{\supply}_{\edge,j}(t)-\rho^{\demand}_{\edge,j}(t)|.
\end{equation}
Functions $t\to \hat{\mathcal L}^1(\rho^\supply(\cdot,t),\rho^\demand(\cdot,t))$ and $t\to \Hnorm(\rho^\supply(\cdot,t),\rho^\demand(\cdot,t))$ are shown in Fig.\ \ref{fig:test1_cfr_L1_vs_W} for two different network size. 
\begin{figure}[h!]
\centerline{
\includegraphics[width=0.45\textwidth]{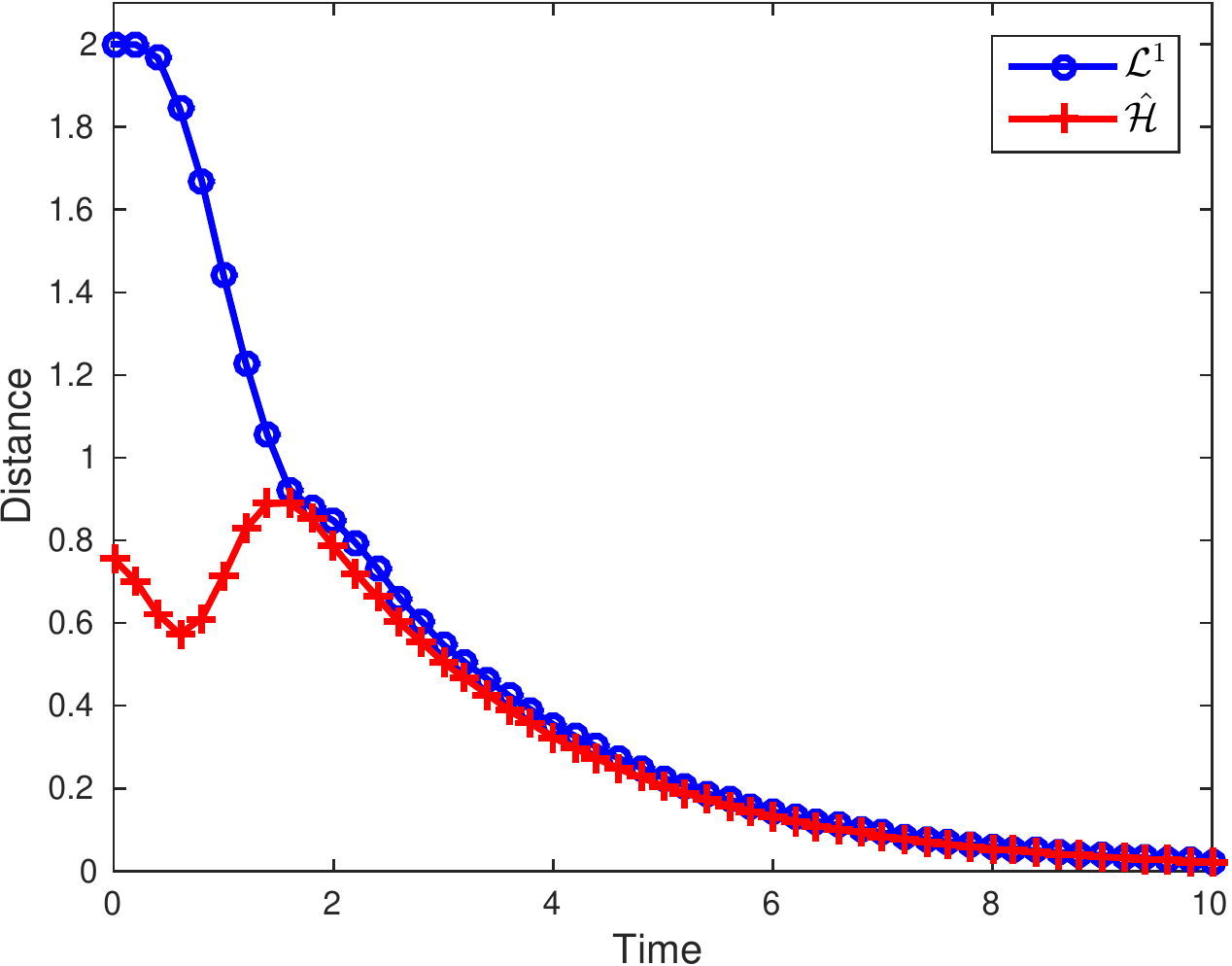} \qquad
\includegraphics[width=0.45\textwidth]{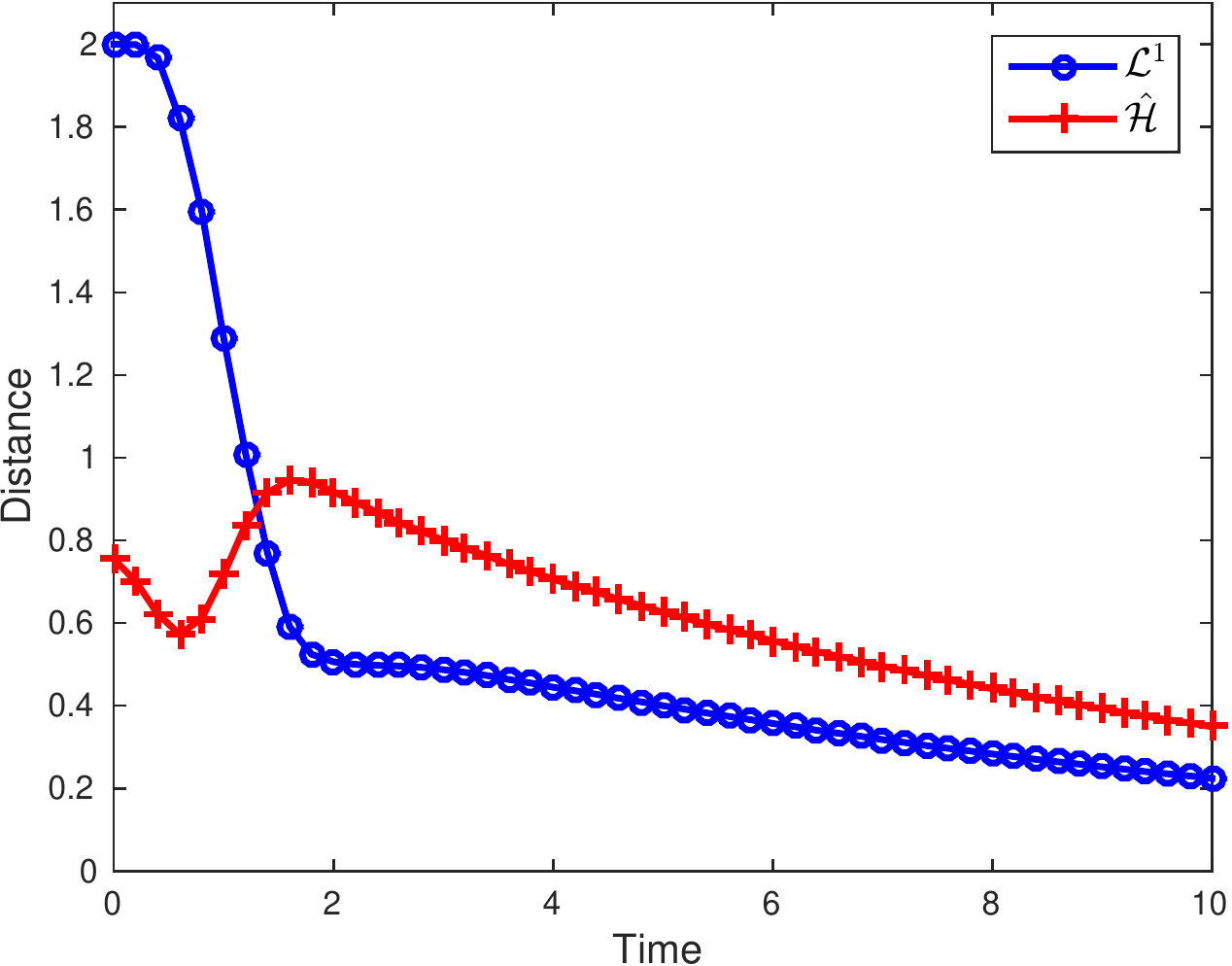}
}
\caption{Sensitivity to initial data ($\hat{\mathcal L}^1$ vs.\ $\Hnorm$). Comparison between functions \mbox{$t\to \hat{\mathcal L}^1(\rho^\supply(\cdot,t),\rho^\demand(\cdot,t))$} and \mbox{$t\to \Hnorm(\rho^\supply(\cdot,t),\rho^\demand(\cdot,t))$} for $\Nnl=3$ (left) and  $\Nnl=5$ (right).} 
\label{fig:test1_cfr_L1_vs_W}
\end{figure}
Initially, the $\hat{\mathcal L}^1$ distance shows a plateau, which lasts until the supports of the densities $\rho^\supply$ and $\rho^\demand$ are disjoint. This is not the case of the Wasserstein distance which instead immediately decreases. 
After that, the supports of the two densities start to overlap but the regions with maximal density move away from each other, see Fig.\ \ref{fig:test1_cfr_L1_vs_W_densities}. When this process ends, we get the maximal value of the Wasserstein distance and the change of slope of the $\hat{\mathcal L}^1$ distance. 
Later on, the two densities uniformly distribute along the network and the two distances go smoothly to 0.
\begin{figure}[h!]
\centerline{
\includegraphics[width=0.41\textwidth]{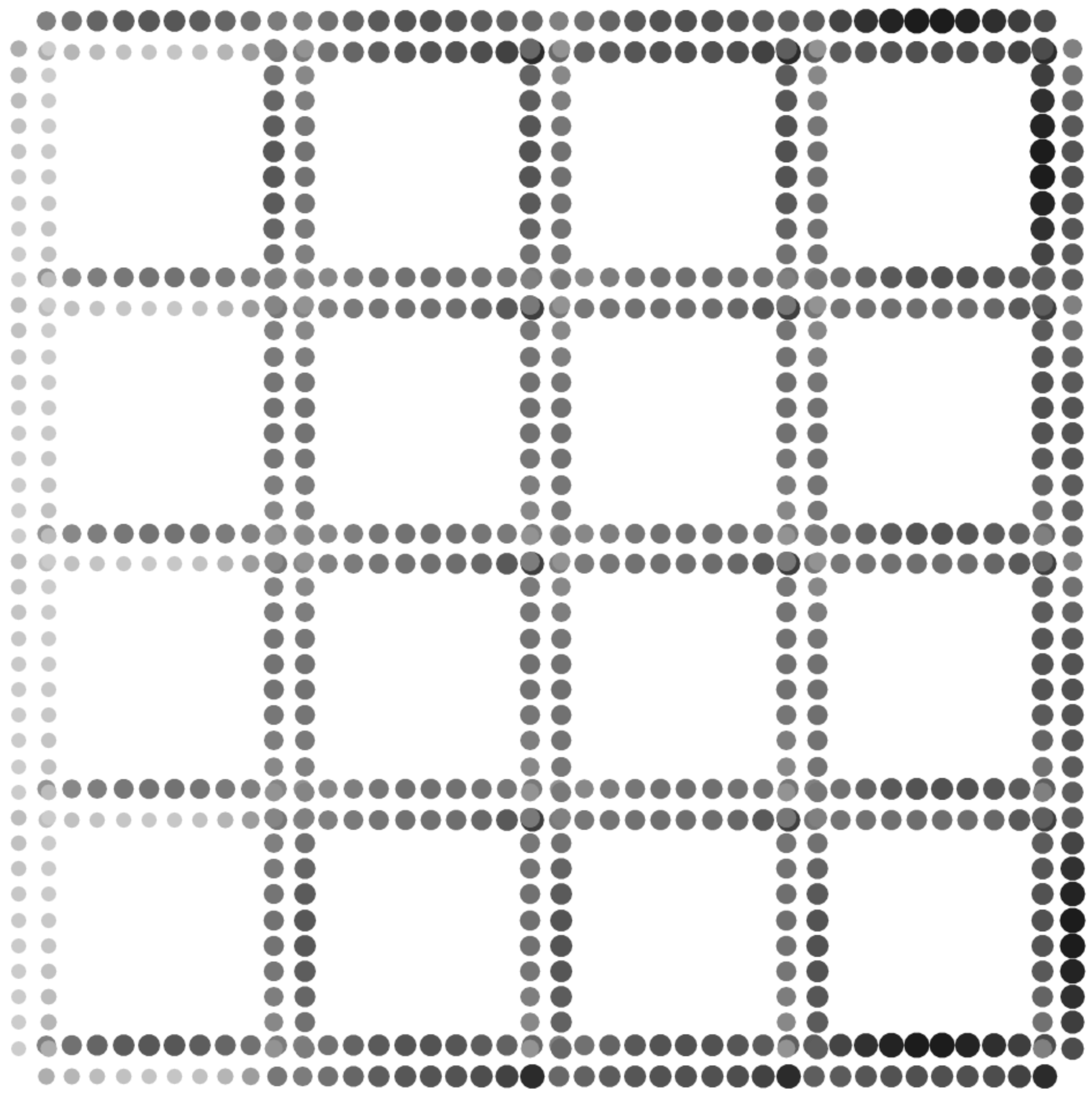} \qquad\quad
\includegraphics[width=0.445\textwidth]{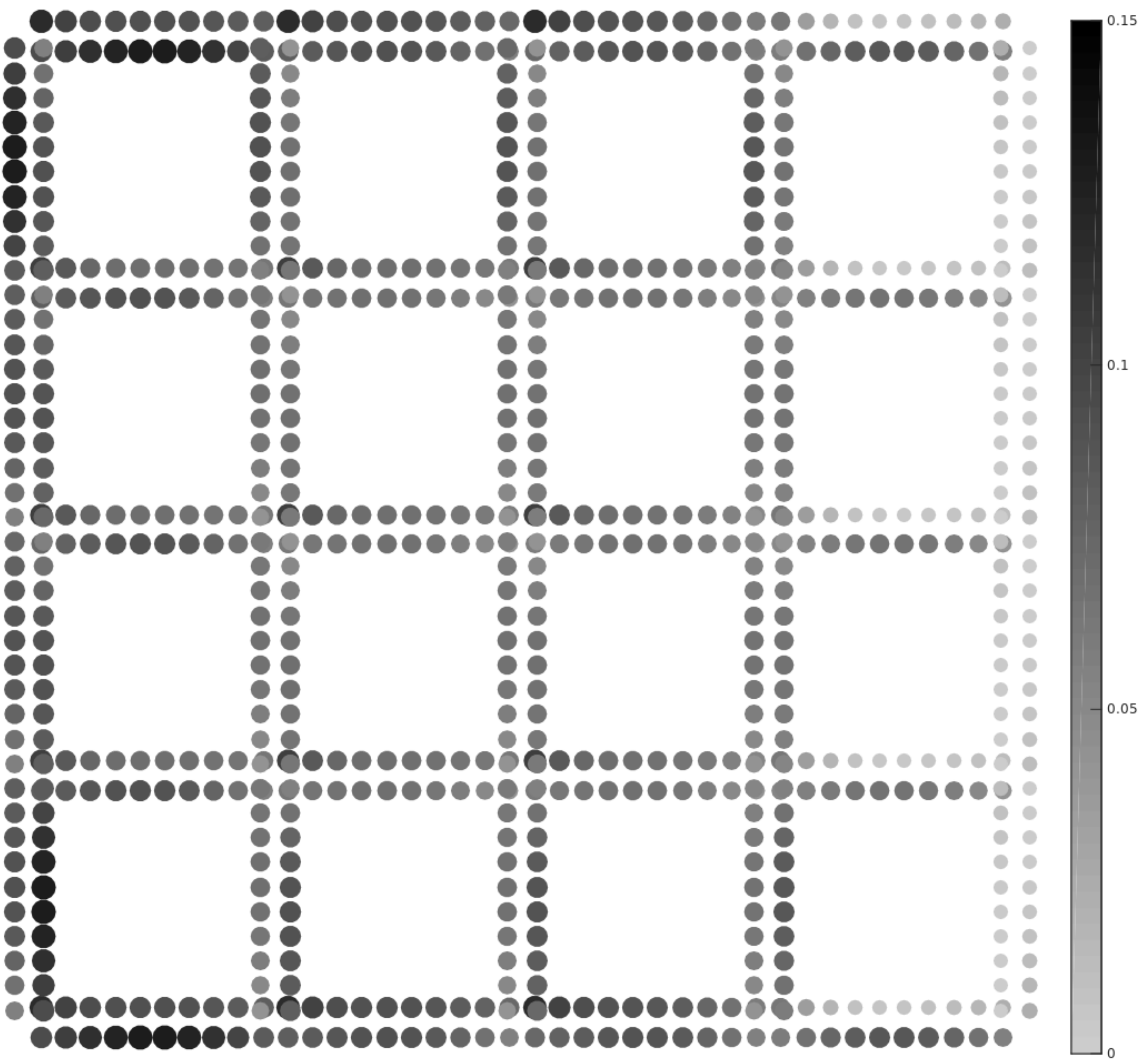}
}
\caption{Sensitivity to initial data ($\hat{\mathcal L}^1$ vs.\ $\Hnorm$). $\rho^\supply$ (left) and $\rho^\demand$ (right) at time $T=1.8$.} 
\label{fig:test1_cfr_L1_vs_W_densities}
\end{figure}

\subsubsection{Numerical convergence as $\Dx\rightarrow 0$.}\label{sec:num_conv_Dx}
In this test we consider a small network ($\Nnl=3$) and we compute the Wasserstein distance $\Hnorm(\rho^\supply,\rho^\demand)$ for different values of $\Nce$. 
Fig.\ \ref{fig:test1_convergenza} shows the functions $t\to\Hnorm(\rho^\supply(\cdot,t),\rho^\demand(\cdot,t))$ for $J_\edge=10,20,40,80$ and $\Nce\to\Hnorm(\rho^\supply(\cdot,T),\rho^\demand(\cdot,T))$ at fixed time $T=1.4$. 
\begin{figure}[h!]
\centerline{
\includegraphics[width=0.45\textwidth]{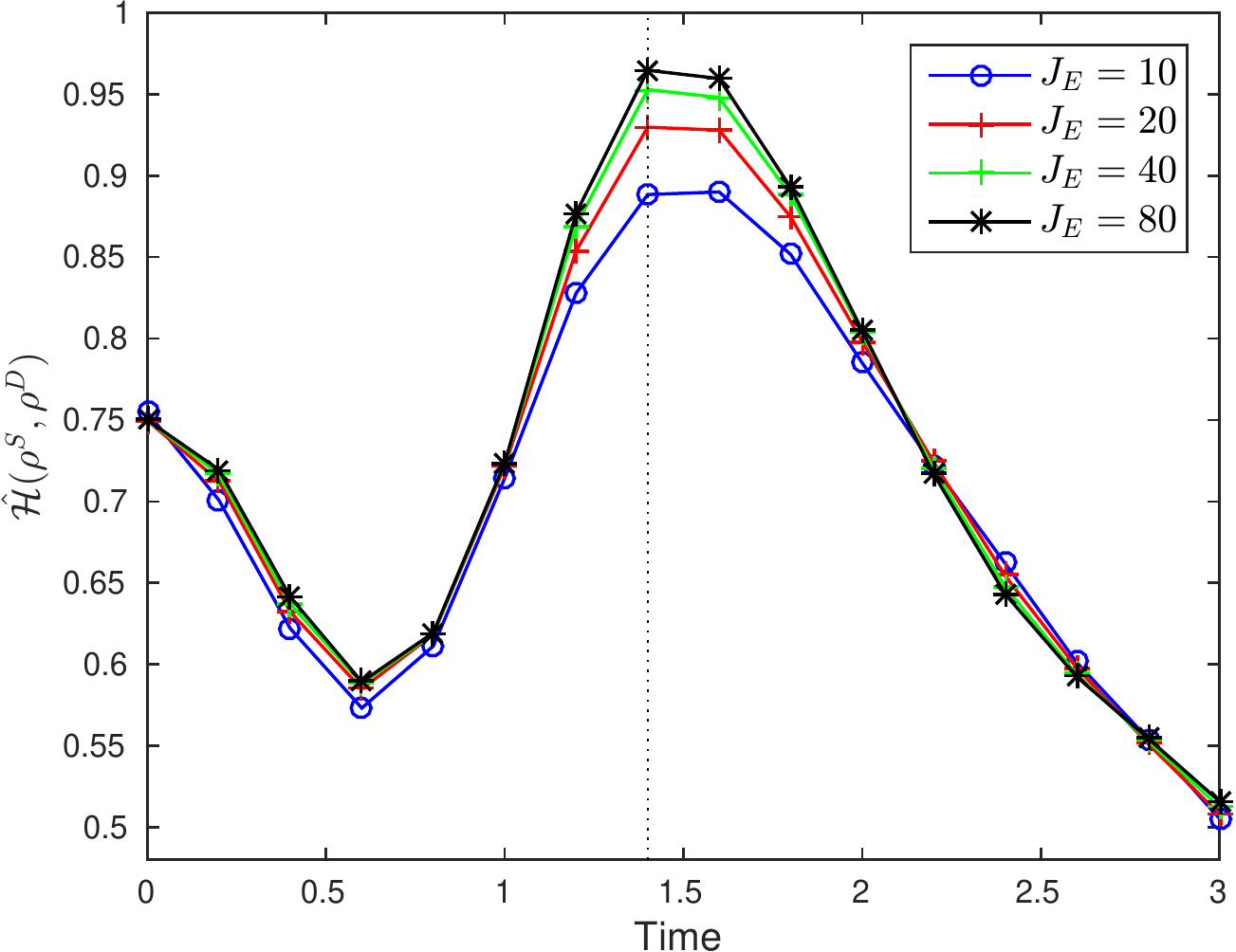} \qquad
\includegraphics[width=0.45\textwidth]{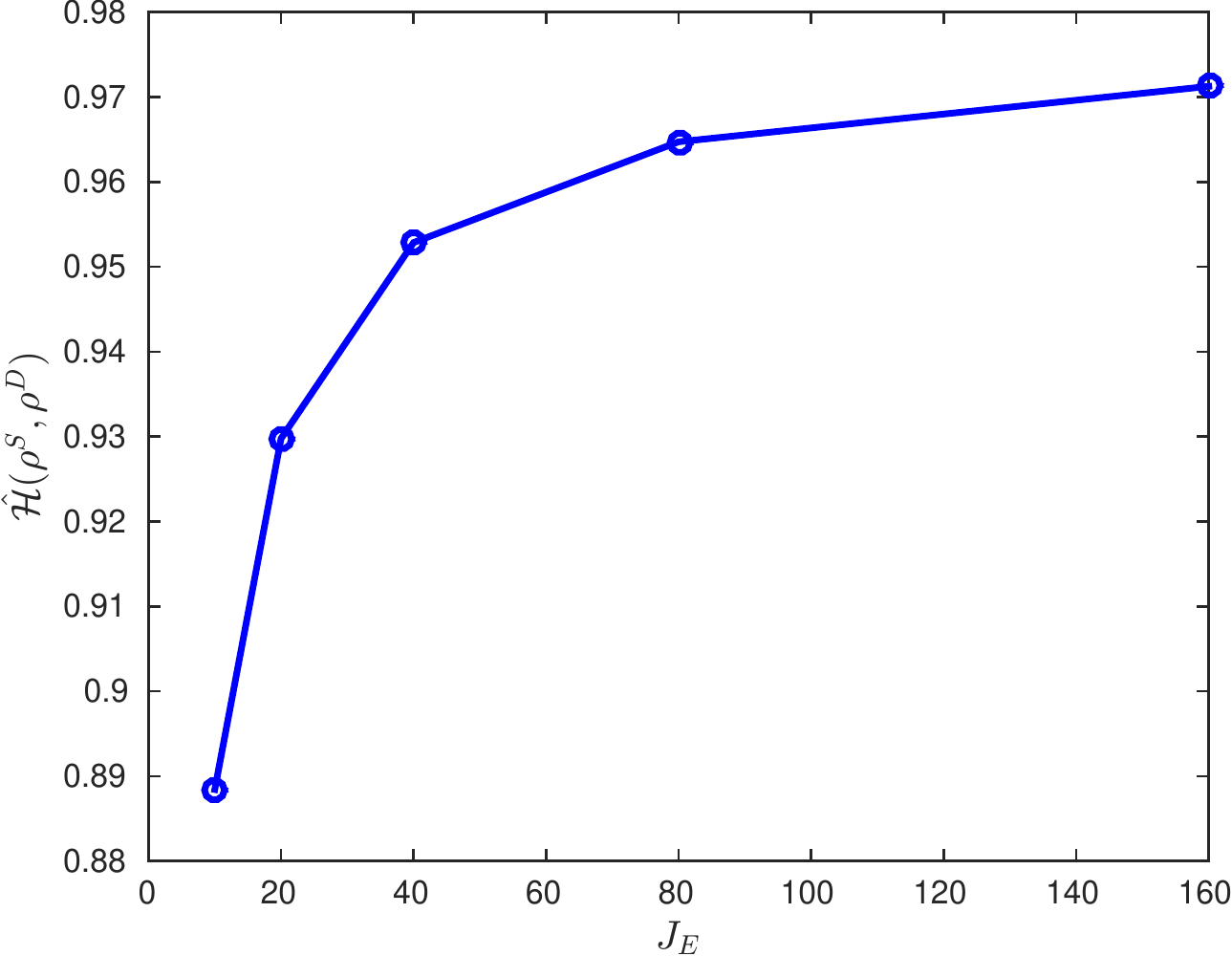}
}
\caption{Sensitivity to initial data (convergence). Function $t\to\Hnorm(\rho^\supply(\cdot,t),\rho^\demand(\cdot,t))$ for different values of $J_\edge$ (left) and function $\Nce\to\Hnorm(\rho^\supply(\cdot,T),\rho^\demand(\cdot,T))$ with $T=1.4$ (right).} 
\label{fig:test1_convergenza}
\end{figure}
Fig.\ \ref{fig:test1_convergenza}(left) suggests a relatively small sensitivity to the space step. We can safely assume that the difference between the values of $\Hnorm$ obtained with $\Nce=10$ and $\Nce=160$ is lower than 10\% with respect to the largest of the two values.
We get similar results also for larger networks. 
The numerical convergence of $\Hnorm=\Hnorm(\Nce)$ as $\Nce\to +\infty$ is also clearly visible in Fig.\ \ref{fig:test1_convergenza}(right).

In the next sections, the sensitivity analysis will be obtained with $\Nce=10$ which seems to be a good compromise between accuracy of the results and computational costs.

\subsection{Sensitivity to fundamental diagram.}\label{sec:sens.FD}
In this test we measure the sensitivity to the two parameters of the fundamental diagram, namely $\sigma$ and $\fmax$. 
The goal is to quantify the impact of a possible error in measuring the capacity of the roads or in describing the drivers behavior. \REV{Note that the linear structure of the fundamental diagram used here (see Fig.\ \eqref{fig:FD}) does not play any special role and any other fundamental diagram could be considered, as long as it is duly parametrized.}

The parameters which remain fixed in this test are
\begin{itemize}
\item \emph{Initial density}: 
$$
\rho^0_{\edge,j} = \left\{\begin{array}{ll}
0.5 &\mbox{ on rightward roads},\\
\smallskip
0 & \mbox{ elsewhere},
\end{array}\right.  \qquad
\edge\in\mathcal{E},\quad j=1\ldots,\Nce.
$$
\item \emph{Distribution matrix}: 
$$
\alpha^\vertex_{\textsc{r} \textsc{r}'}=\frac{1}{\Nvout}, \qquad \forall \vertex\in\mathcal V,\quad  \textsc{r}=1,\ldots,\Nvinc, \quad \textsc{r}'=1,\ldots,\Nvout.
$$ 
\end{itemize}

In Fig.\ \ref{fig:test2_DF}(left) we report the distance between the solutions $\rho^\supply$ and $\rho^\demand$ at time $T=20$ obtained with $\fmax^\supply=0.25$, $\sigma^\supply=0.3$ and $\fmax^\demand=0.25$, $\sigma^\demand\in[0.15,0.5]$, respectively.  
In Fig.\ \ref{fig:test2_DF}(right) we report the distance between the two solutions at time $T=20$ obtained with $\sigma^\supply=0.3$, $\fmax^\supply=0.25$, and $\sigma^\demand=0.3$, $\fmax^\demand\in[0.15,0.4]$, respectively.
\begin{figure}[h!]
\begin{center}
\begin{tabular}{lr}
\psfig{file=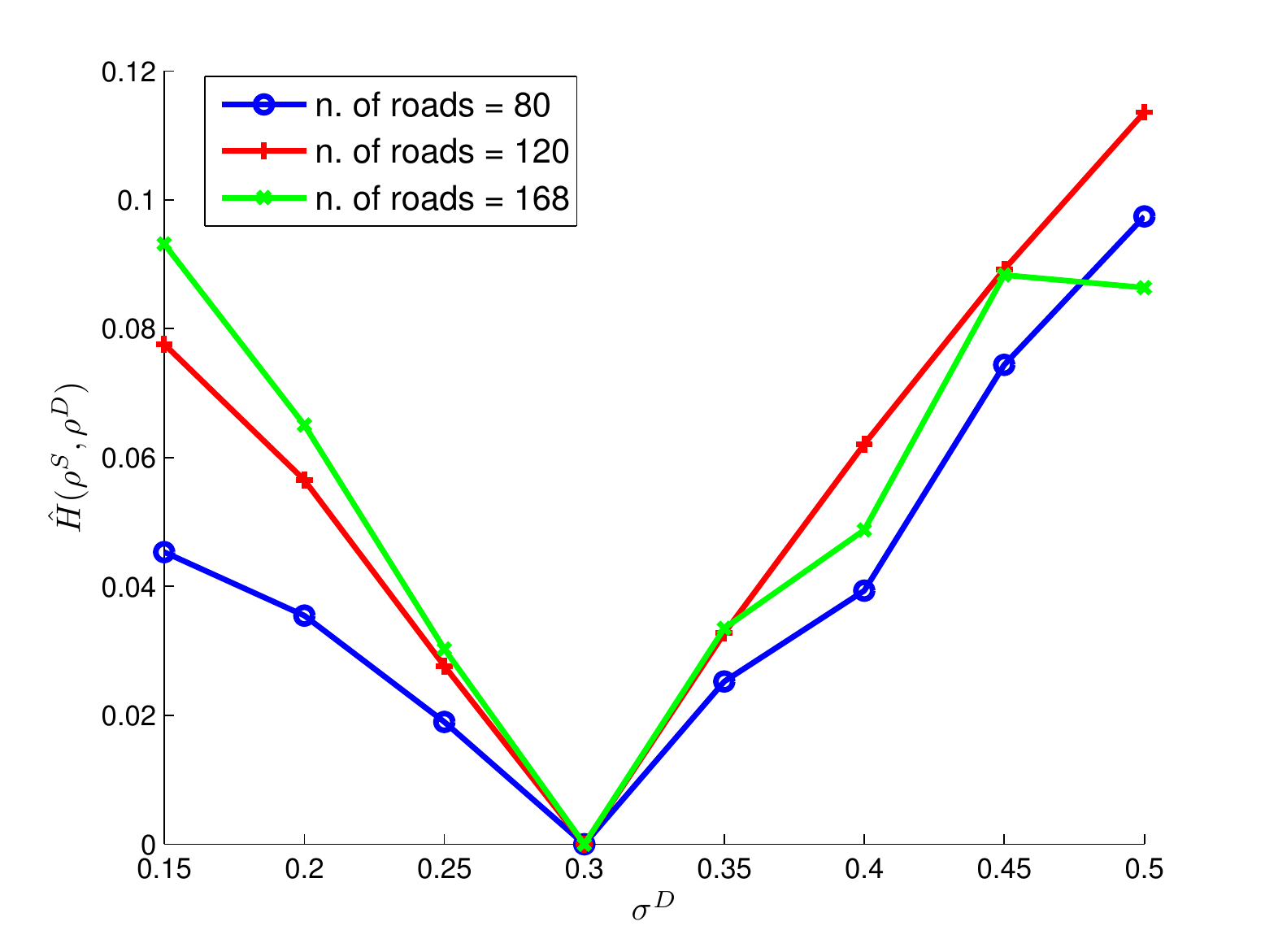,width=2.4in} &
\psfig{file=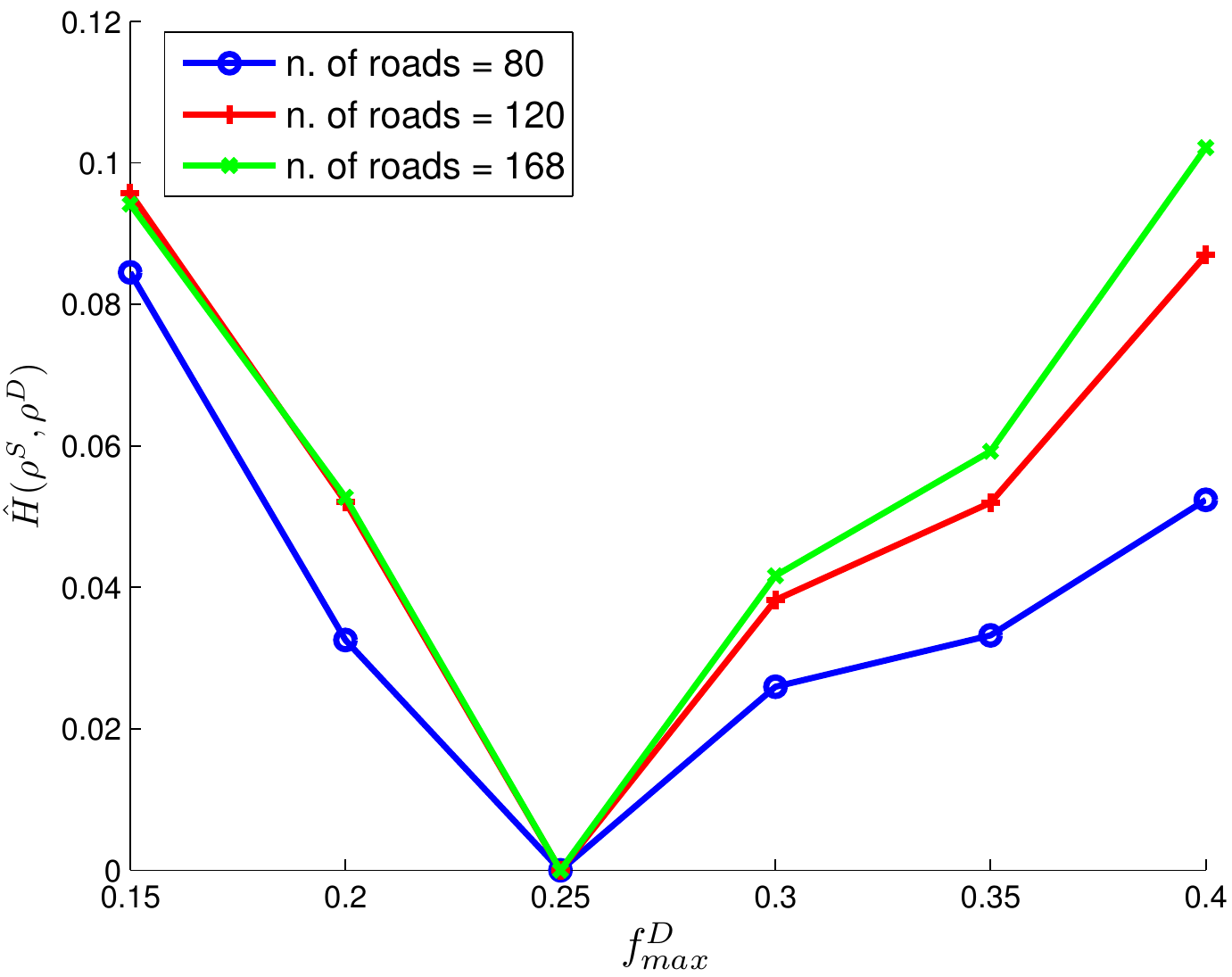,width=2.4in}
\end{tabular}
\end{center}
\caption{Sensitivity to fundamental diagram. Function \mbox{$\sigma^\demand\to \Hnorm(\rho^\supply(\cdot,T),\rho^\demand(\cdot,T))$} (left) and, \mbox{$\fmax^\demand\to \Hnorm(\rho^\supply(\cdot,T),\rho^\demand(\cdot,T))$} (right), for $\Nnl=5,6,7$.} 
\label{fig:test2_DF}
\end{figure}
Errors in the calibration of $\sigma$ or $\fmax$ lead to similar discrepancies, which are again amplified by the network size. Discrepancies grow approximately linearly with respect to both $|\sigma^\demand-\sigma^\supply|$ and $|\fmax^\demand-\fmax^\supply|$.

In Fig.\ \ref{fig:test2_DF_WT0} we report the distance between the solutions $\rho^\supply$ and $\rho^\demand$ obtained with $\fmax^\supply=\fmax^\demand=0.25$, $\sigma^\supply=0.3$, $\sigma^\demand=0.2$ (left), and $\sigma^\supply=\sigma^\demand=0.3$, $\fmax^\supply=0.25$, $\fmax^\demand=0.3$ (right), as a function of time.
\begin{figure}[htp]
\begin{center}
\begin{tabular}{cc}
\psfig{file=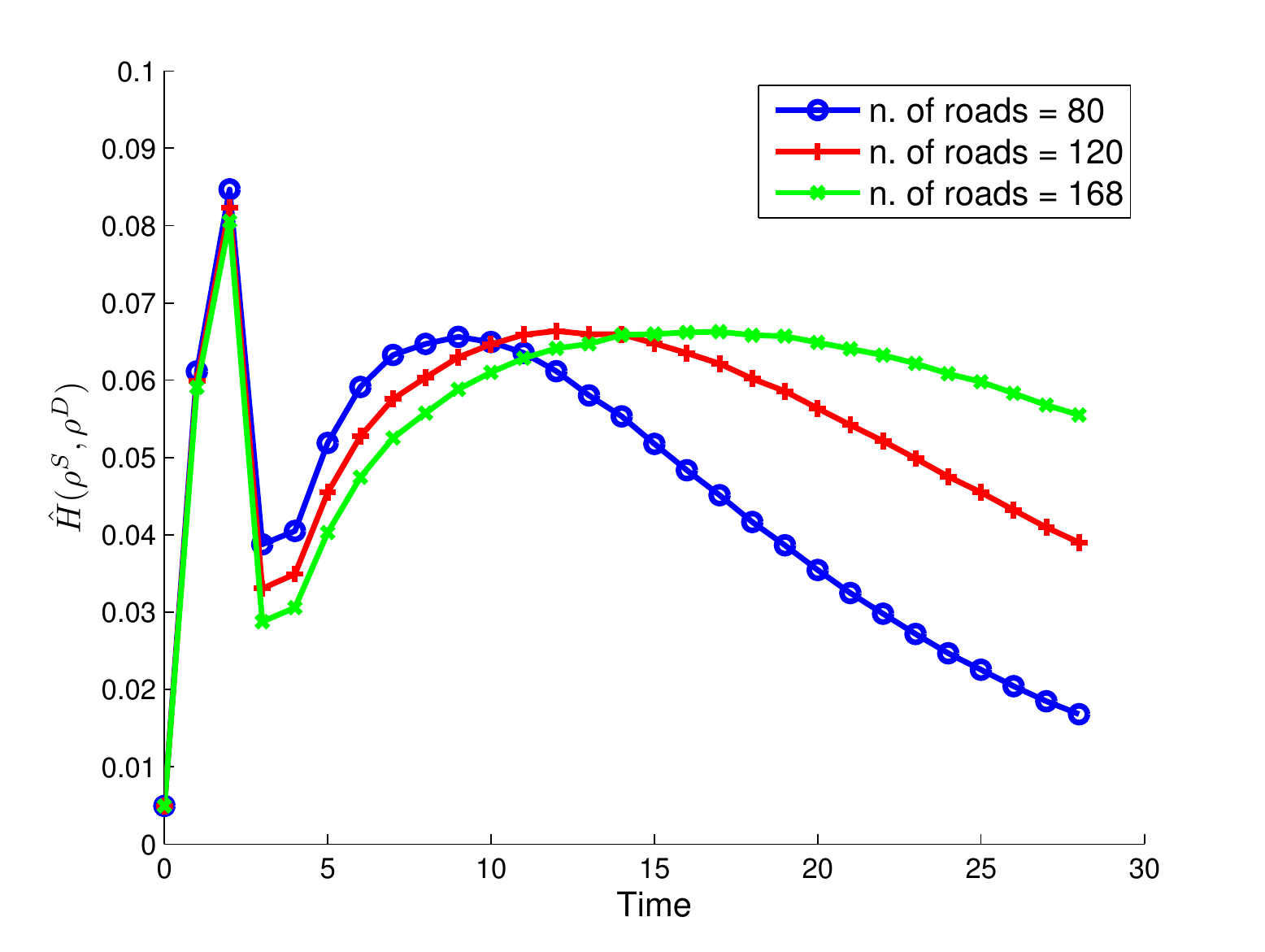,width=2.4in}
&
\psfig{file=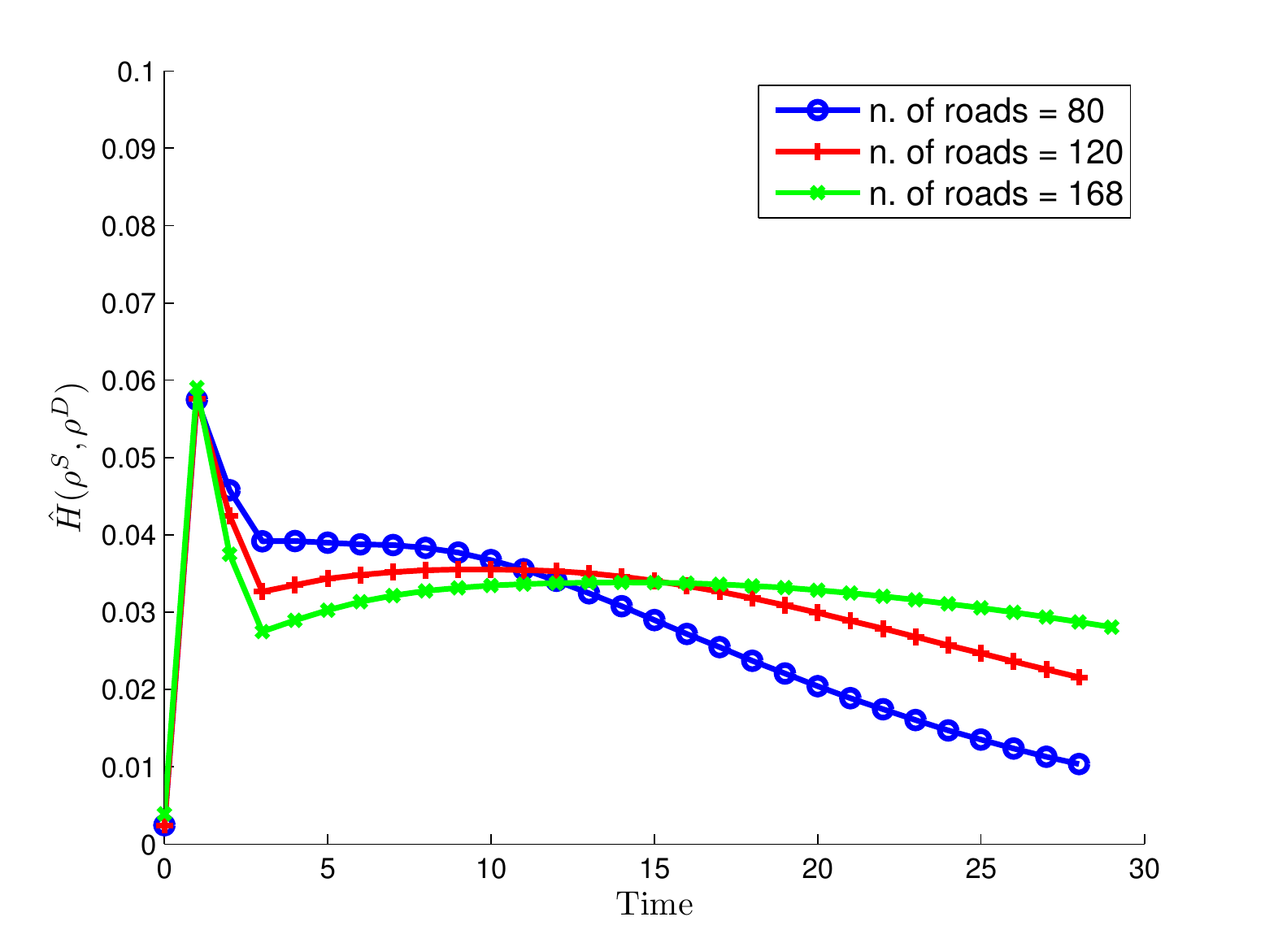,width=2.4in}
\end{tabular}
\end{center}
\caption{Sensitivity to fundamental diagram. With respect to $\sigma$ (left) and $\fmax$ (right). Function \mbox{$t\to \Hnorm(\rho^\supply(\cdot,t),\rho^\demand(\cdot,t))$} for $\Nnl=5,6,7$.} 
\label{fig:test2_DF_WT0}
\end{figure}
Again, we see that the distances tend to 0 as $t \to\infty$ because vehicles spread across the networks toward a constant stationary density distribution (see Remark \ref{rem:tuttoazero}), and the size of the network affects the time scale only.

\subsection{Sensitivity to the distribution matrix.}\label{sec:sens.alpha}
In this test we measure the sensitivity to the distribution coefficients at junctions, see Sec.\ \ref{sec:themodel}. 
The goal is to quantify the impact of a possible error in the knowledge of the path choice at junctions.

The parameters which remain fixed in this test are
\begin{itemize}
\item \emph{Initial density}: $\rho^0_{\edge,j} =0.5$, \quad $\edge\in\mathcal{E}$, $j=1,\ldots,\Nce$. 
\item \emph{Fundamental diagram}: $\sigma=0.3$ and $\fmax=0.25$.
\end{itemize}

Supply distribution $\rho^\supply$ is obtained by means of equidistributed coefficients
$$
\alpha^{\supply,\vertex}_{\textsc{r} \textsc{r}'}=\frac{1}{\Nvout}, \quad\forall \vertex\in\mathcal V,\quad \textsc{r}=1,\ldots,\Nvinc, \quad \textsc{r}'=1,\ldots,\Nvout. 
$$
Note that, due to the symmetry of the network and the initial datum, $\rho^\supply\equiv 0.5$ for all $x$ and $t$.

\subsubsection{Single junction.}\label{sec:singlejunctions}
Here demand distribution $\rho^\demand$ is obtained by varying the distribution coefficients at the junction $\bar\vertex$ located at the very center of the network (see, e.g., vertex 13 in Fig.\ \ref{fig:manhattan_vuota}). Variation is performed by means of a scalar parameter $\varepsilon>0$. We have, for all incoming roads $\textsc{r}=1,2,3,4$,
$$
\alpha^{\demand,\bar\vertex}_{\textsc{r} 1}=\frac{1}{\Nvout}+\varepsilon,\quad
\alpha^{\demand,\bar\vertex}_{\textsc{r} 2}=\frac{1}{\Nvout}-\varepsilon,\quad
\alpha^{\demand,\bar\vertex}_{\textsc{r} 3}=\frac{1}{\Nvout}+\varepsilon,\quad
\alpha^{\demand,\bar\vertex}_{\textsc{r} 4}=\frac{1}{\Nvout}-\varepsilon.
$$

In Fig.\ \ref{fig:test3_rhoD} we report the distribution $\rho^\demand$ at time $t=5$ and $t=45$ obtained with $\varepsilon=0.1$, to be compared with the constant distribution $\rho^\supply\equiv 0.5$.
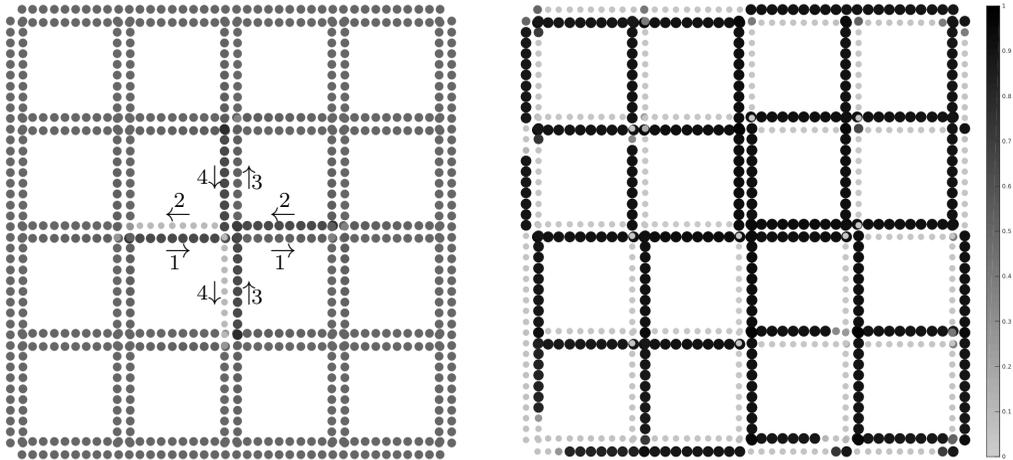
\begin{figure}[h!]
\centerline{
\begin{overpic}[width=0.45\textwidth]{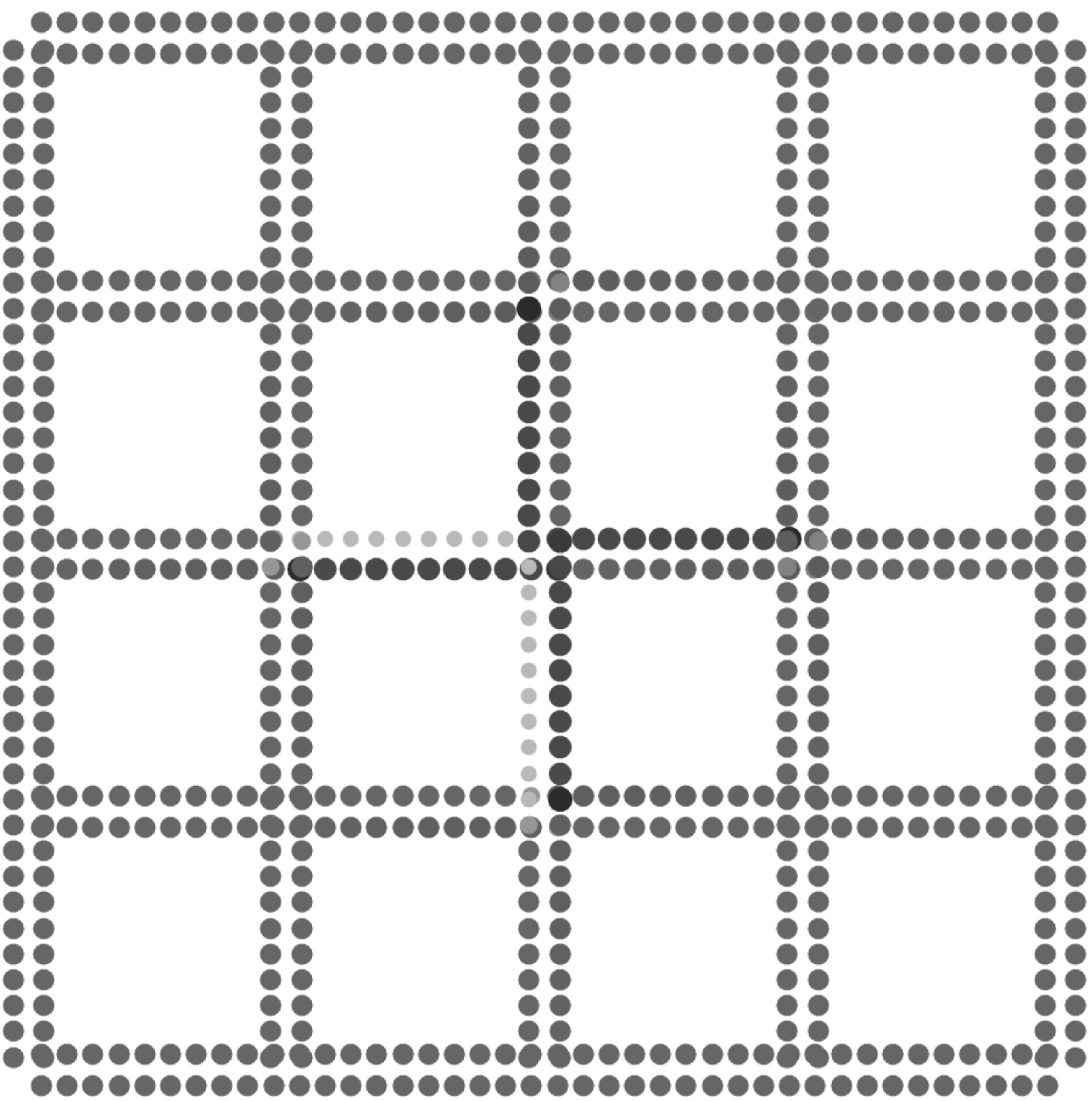}
\put(58,44){$\rightarrow$}\put(59,41.5){\footnotesize 1}
\put(35,44){$\rightarrow$}\put(36,41.5){\footnotesize 1}
\put(58,52){$\leftarrow$}\put(60,55){\footnotesize 2}
\put(35,52){$\leftarrow$}\put(37,55){\footnotesize 2}
\put(52,35){$\uparrow$}\put(54,34){\footnotesize 3}
\put(44.5,35){$\downarrow$}\put(42,35){\footnotesize 4}
\put(52,60){$\uparrow$}\put(54,59){\footnotesize 3}
\put(44.5,60){$\downarrow$}\put(42,60){\footnotesize 4}
\end{overpic}
\qquad
\includegraphics[width=0.49\textwidth]{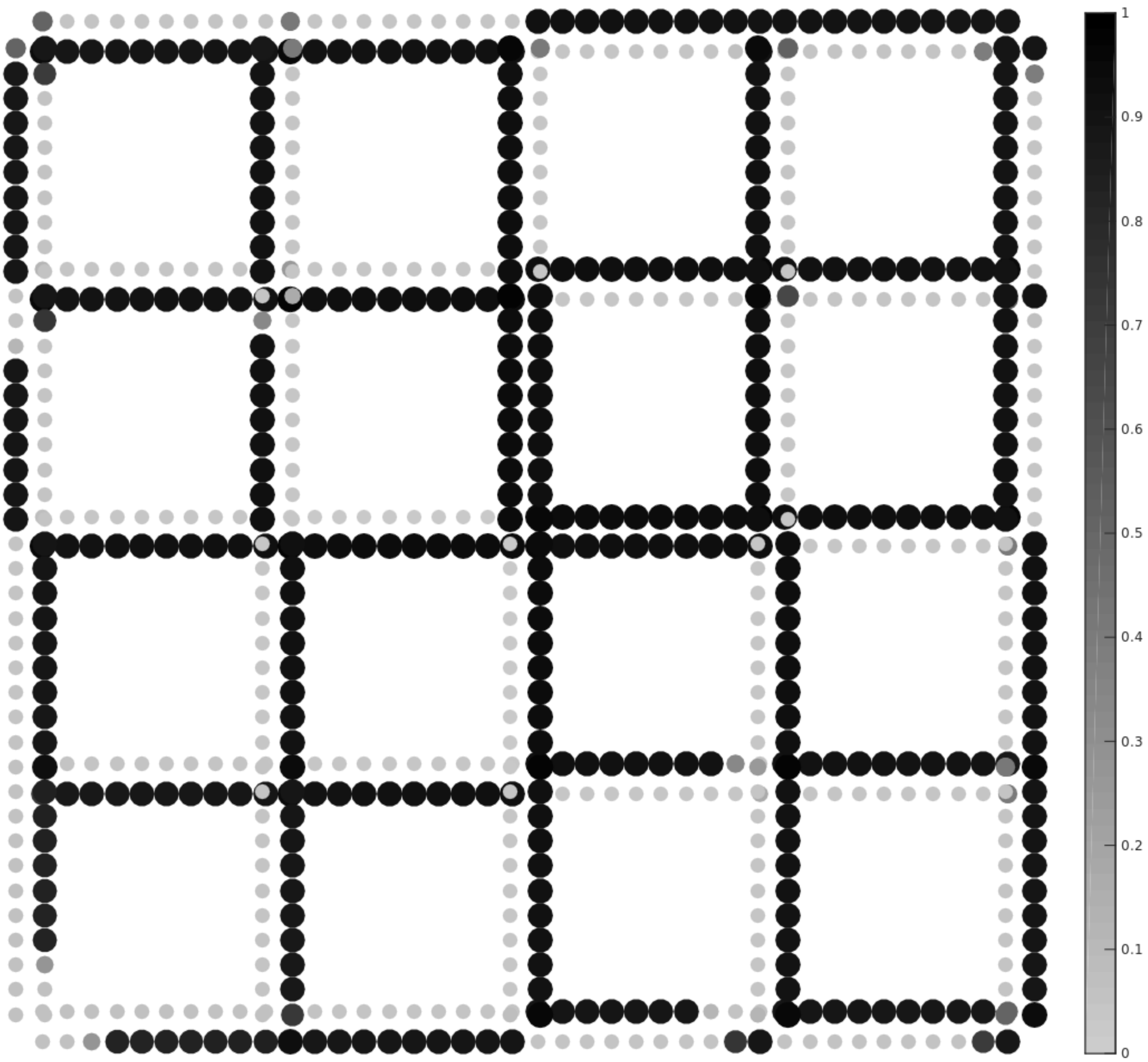}
}
\caption{Sensitivity to distribution matrix (single junction). Density $\rho^\demand$ at time $t=5$ (left) and $t=45$ (right).} 
\label{fig:test3_rhoD}
\end{figure}
Remarkably, a minor local modification of the traffic distribution in a single junction breaks the symmetry and has a great impact on the solution. This time the density does not tend to distribute uniformly across the network and then we expect the distance $W(\rho^\supply,\rho^\demand)$ to increase in time, although the growth cannot continue indefinitely since the distance between two distributions on a finite network is finite.

In Fig.\ \ref{fig:test3_WT_alpha} we show the distance between the two densities as a function of time for $\varepsilon=0.1,\ 0.2$. The distance is indeed increasing and bounded as expected. Moreover  a larger $\varepsilon$ accelerates the growth of the distance. 
Further comments will be given in the following Sec.\ \ref{sec:comparison}.
\begin{figure}[h!]
\begin{center}
\includegraphics[width=0.48\textwidth]{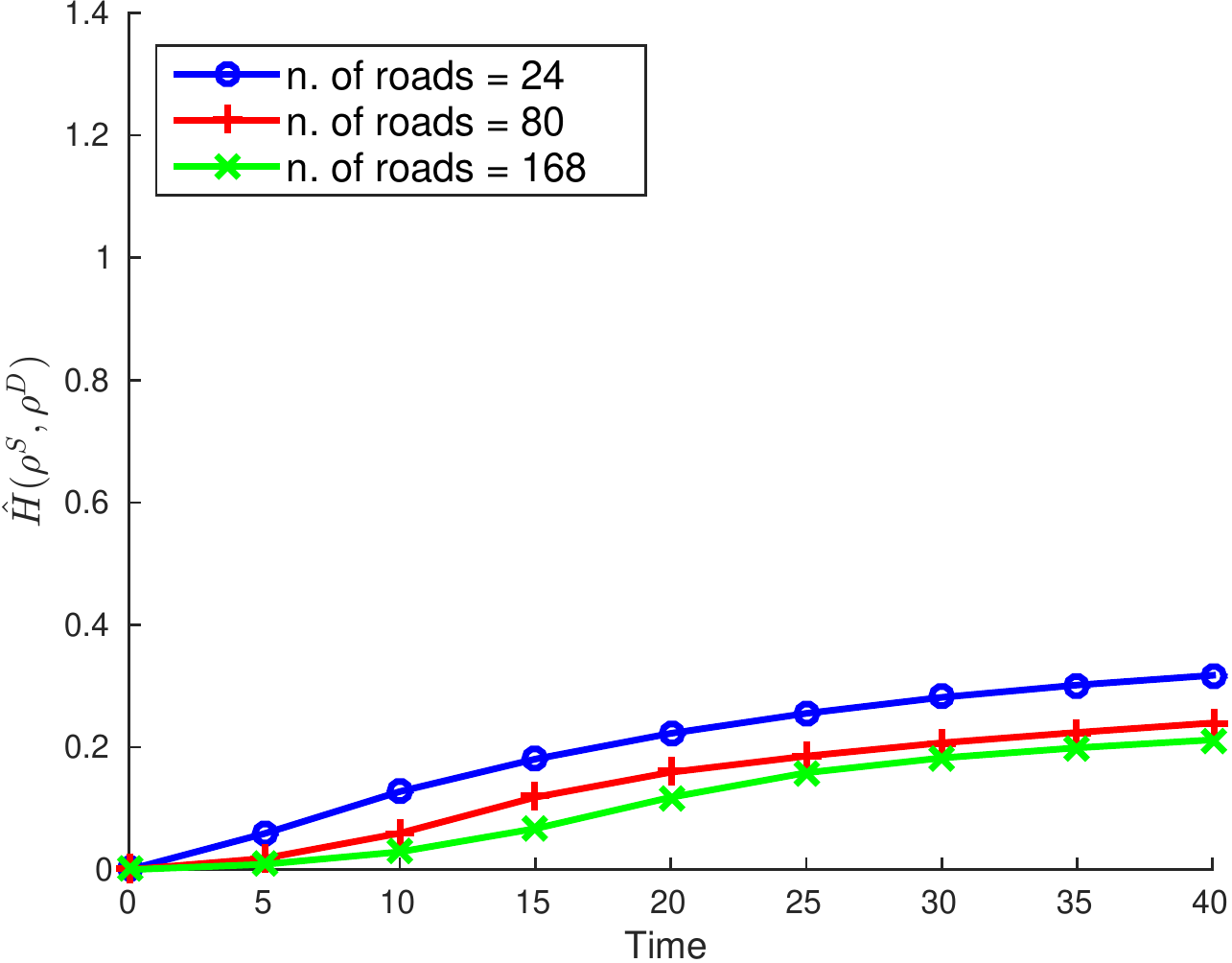}\quad
\includegraphics[width=0.48\textwidth]{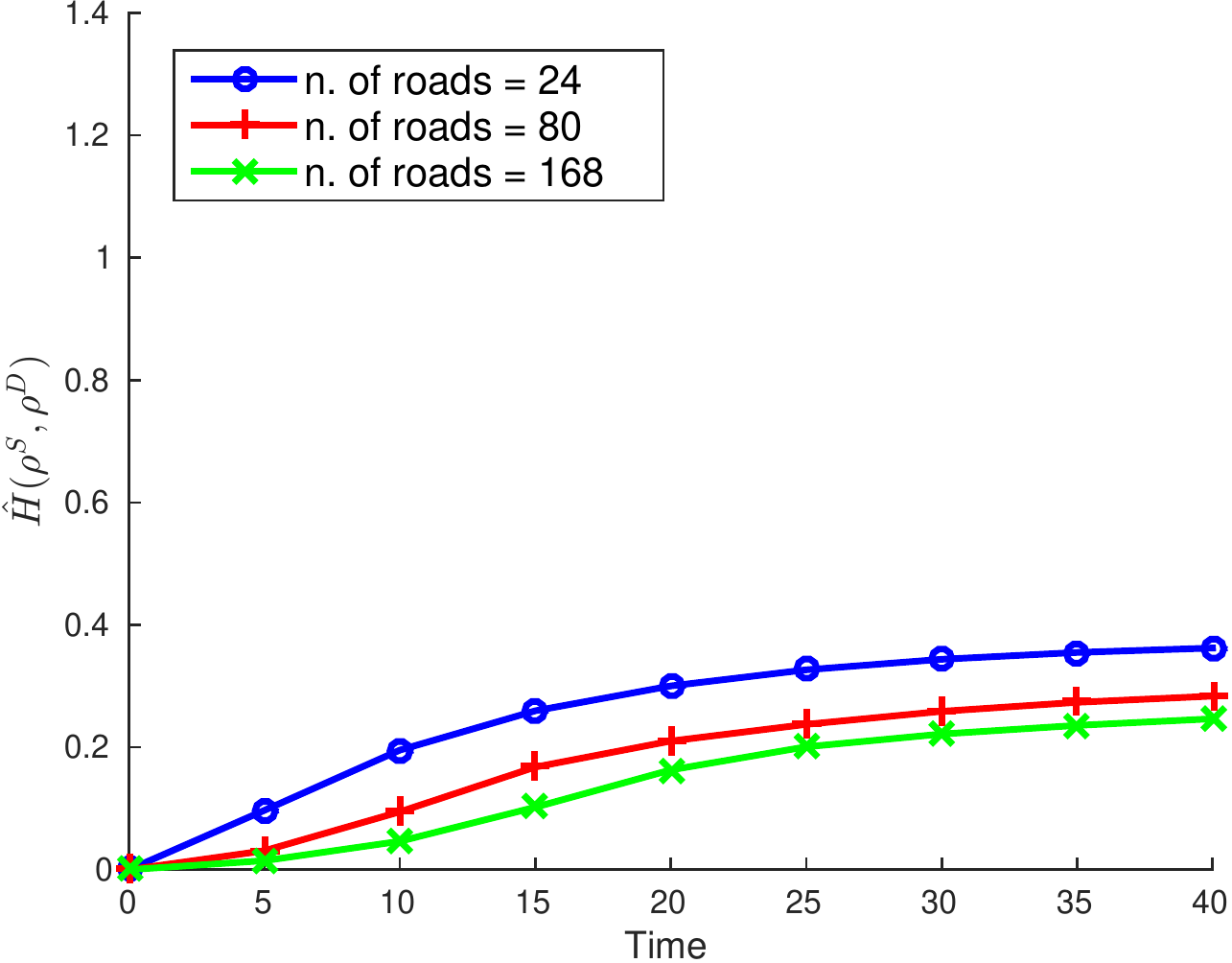}
\end{center}
\caption{Sensitivity to traffic distribution at junctions (single junction). Function \mbox{$t\to \Hnorm(\rho^\supply(\cdot,t),\rho^\demand(\cdot,t))$} for $\Nnl=3,5,7$.  $\varepsilon=0.1$ (left), $\varepsilon=0.2$ (right).} 
\label{fig:test3_WT_alpha}
\end{figure}

\subsubsection{All junctions.}\label{sec:alljunctions}
Let us now modify \emph{all} the distribution coefficients, and not only those at one junction. In the following test we set, for any $\vertex$ and $\textsc{r}=1,2,3,4$,
$$
\alpha^{\demand,\vertex}_{\textsc{r} 1}=\frac{1}{\Nvout}+\varepsilon,\quad
\alpha^{\demand,\vertex}_{\textsc{r} 2}=\frac{1}{\Nvout}-\varepsilon,\quad
\alpha^{\demand,\vertex}_{\textsc{r} 3}=\frac{1}{\Nvout}+\varepsilon,\quad
\alpha^{\demand,\vertex}_{\textsc{r} 4}=\frac{1}{\Nvout}-\varepsilon,
$$
if $\vertex$ is labeled by an odd number and 
$$
\alpha^{\demand,\vertex}_{\textsc{r} 1}=\frac{1}{\Nvout}-\varepsilon,\quad
\alpha^{\demand,\vertex}_{\textsc{r} 2}=\frac{1}{\Nvout}+\varepsilon,\quad
\alpha^{\demand,\vertex}_{\textsc{r} 3}=\frac{1}{\Nvout}-\varepsilon,\quad
\alpha^{\demand,\vertex}_{\textsc{r} 4}=\frac{1}{\Nvout}+\varepsilon,
$$
otherwise (at border junctions only the first two incoming roads $\textsc{r}=1,2$ are considered). 
Results with $\varepsilon=0.1$ are shown in Fig.\ \ref{fig:test3_all_alphas}.
\begin{figure}[h!]
\centerline{
\includegraphics[width=0.45\textwidth]{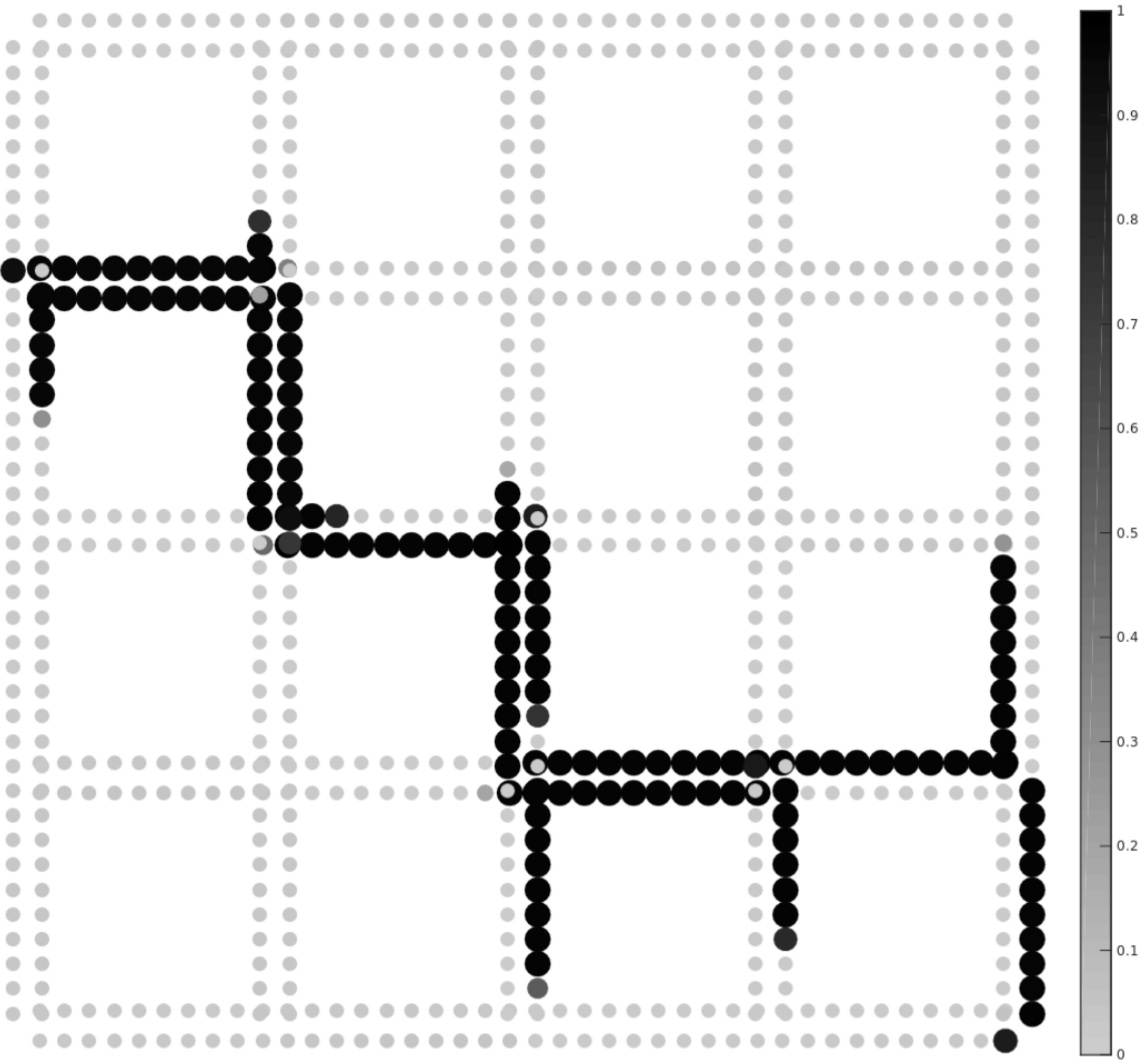}\qquad
\includegraphics[width=0.5\textwidth]{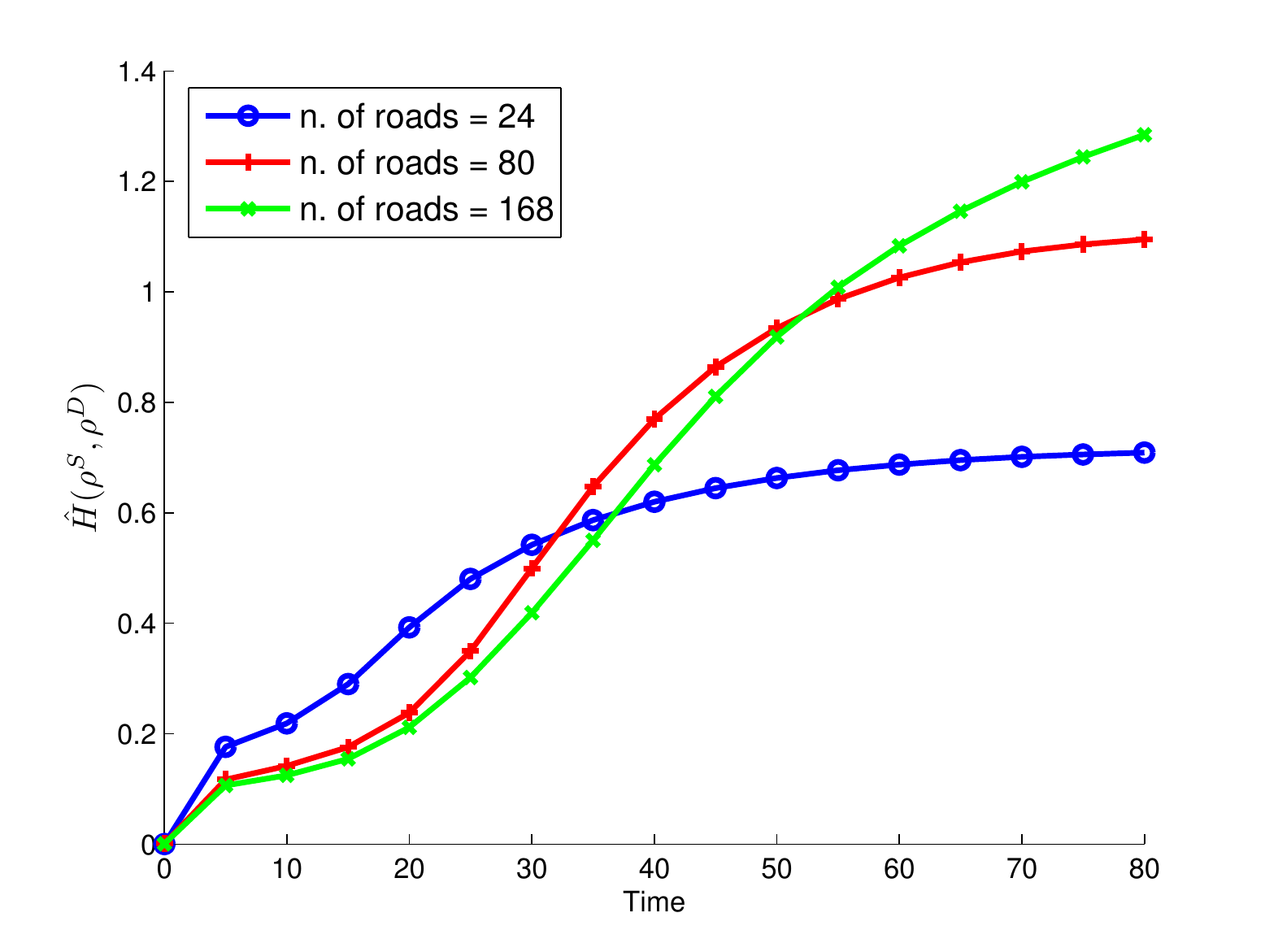}
}
\caption{Sensitivity to traffic distribution at junctions (all junctions). Density $\rho^\demand$ at time $t=55$ (left) and 
function \mbox{$t\to \Hnorm(\rho^\supply(\cdot,t),\rho^\demand(\cdot,t))$} for $\Nnl=3,5,7$ (right).} 
\label{fig:test3_all_alphas}
\end{figure}

\subsubsection{Comparison}\label{sec:comparison}
We observe a great difference between the density distributions $\rho^\demand$'s reported in Secs.\ \ref{sec:singlejunctions} and \ref{sec:alljunctions}. 
In Fig.\ \ref{fig:test3_rhoD}(right) we see that free and congested roads segregate but remain close to each other. 
On the contrary, in Fig.\ \ref{fig:test3_all_alphas}(left) free and congested roads segregate \emph{and separate spatially} from each other.
The Wasserstein distance is able to catch this difference. Indeed, in the former test the Wasserstein distance is almost independent of the network size (Fig.\ \ref{fig:test3_WT_alpha}(left)), while in the latter test (Fig.\ \ref{fig:test3_all_alphas}(right)) it is proportional to the network size (at large times).

\subsection{Sensitivity to road network.}\label{sec:sens.network}
In this test we measure the sensitivity to the road network. 
The goal is to quantify the impact of a possible change in the network, specifically a road closure.

The parameters which remain fixed in this test are
\begin{itemize}
\item \emph{Initial density}: $\rho^0_{\edge,j} =0.3$, \quad $\edge\in\mathcal{E}$, $j=1,\ldots,\Nce$. 
\item \emph{Fundamental diagram}: $\sigma=0.3$ and $\fmax=0.25$.
\item \emph{Distribution matrix:} equidistributed along outgoing roads.
\end{itemize}

Supply distribution $\rho^\supply$ is obtained solving the equations on the complete network, while demand distribution $\rho^\demand$ is obtained by closing the central rightward road $\bar\edge$ (see, e.g., edge 11 in Fig.\ \ref{fig:manhattan_vuota}) just after the initial time, i.e.\ vehicles can come out of the road but none of them can enter. 
Note that, due to the symmetry of the network and the initial datum, $\rho^\supply\equiv 0.3$ for all $x$ and $t$.

In Fig.\ \ref{fig:test4}(left) we report the distribution $\rho^\demand$ at time $t=55$, to be compared with the constant distribution.
\begin{figure}[htp]
\centerline{
\begin{overpic}[width=0.45\textwidth]{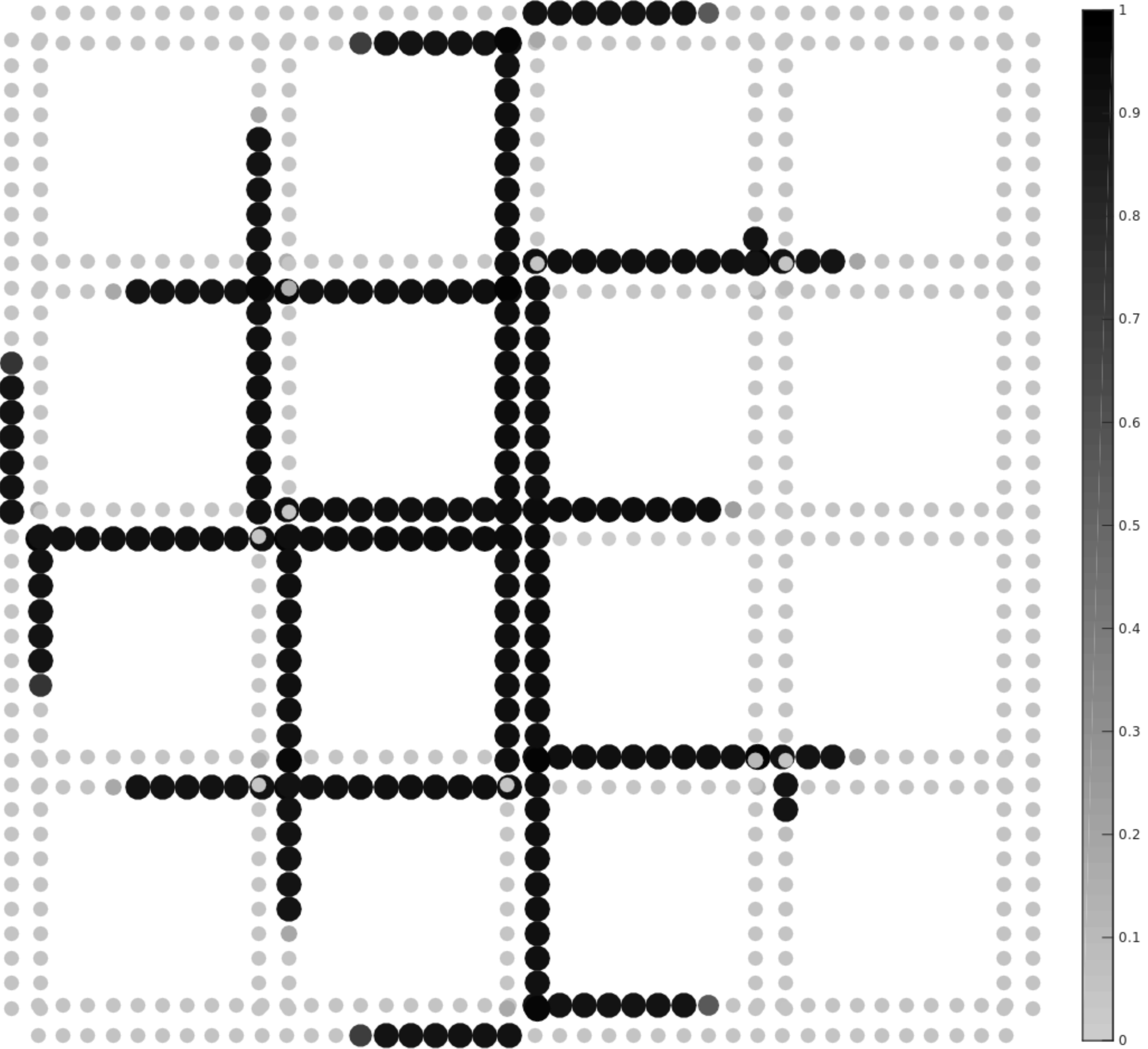}
\put(54,38){\noway}
\end{overpic}
\quad
\includegraphics[width=0.53\textwidth]{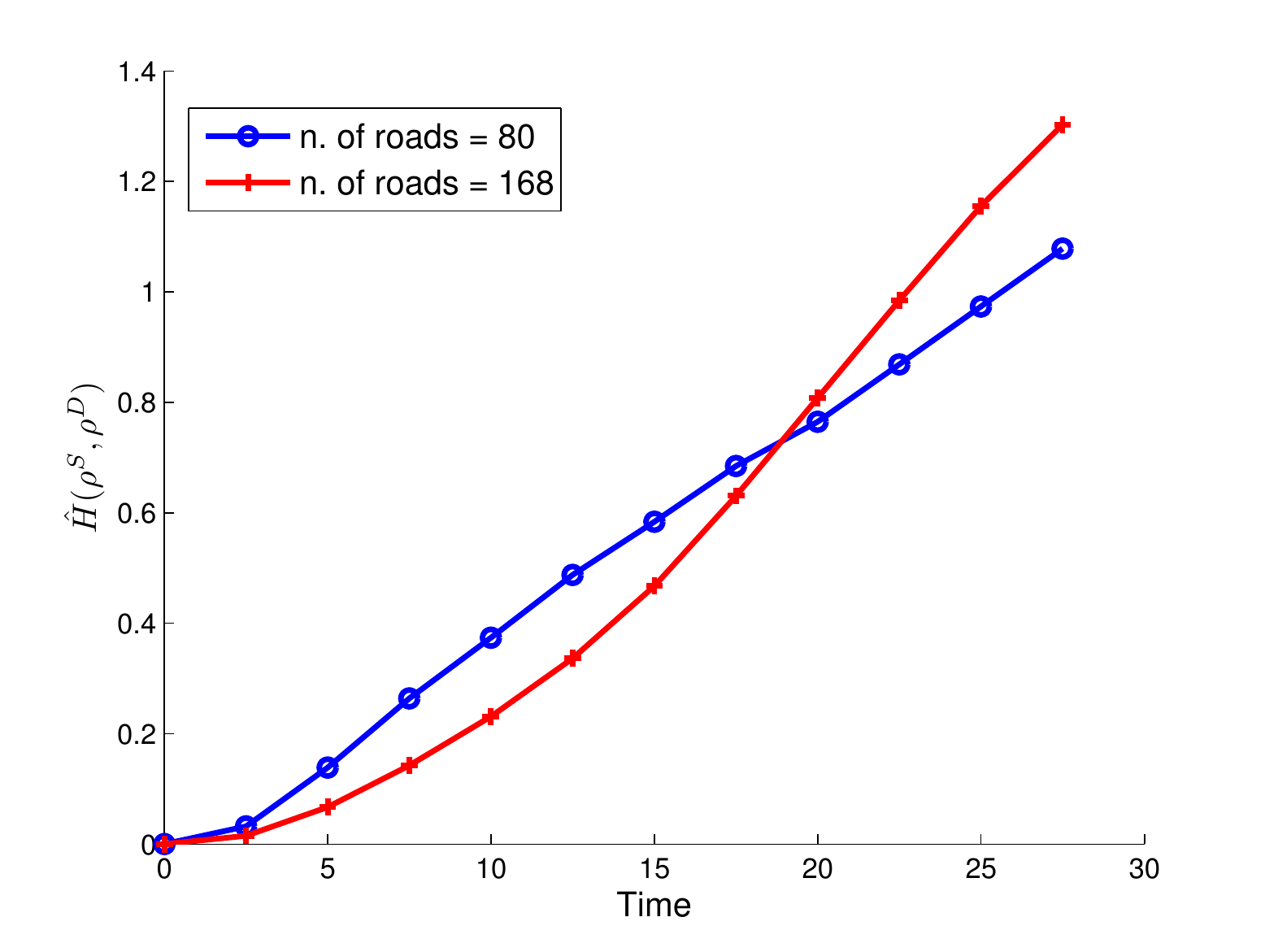}
}
\caption{Sensitivity to road network. Density $\rho^\demand$ at time $t=55$ (left) and 
function \mbox{$t\to \Hnorm(\rho^\supply(\cdot,t),\rho^\demand(\cdot,t))$} for $\Nnl=5,7$ (right).} 
\label{fig:test4}
\end{figure}
We see that the closure of a single road has a great impact on the solution. 
In Fig.\ \ref{fig:test4}(right) we show the distance between the two densities as a function of time. 
The long-time behavior of the sensitivity is proportional to the network size. 
Even if the road closure is a \emph{local} modification of the network dynamics, the behavior with respect to the network size is more similar to that shown in Fig.\ \ref{fig:test3_all_alphas}(right) (all junctions perturbation) than that shown in Fig.\ \ref{fig:test3_WT_alpha}(left) (single junction perturbation).
Again the reason can be found by observing the densities: Fig.\ \ref{fig:test4}(left) shows that the free and congested roads segregate and separate spatially from each other, as in Fig.\ \ref{fig:test3_all_alphas}(left).

\section{Conclusions.}\label{sec:conclusions}
In this paper we have studied the impact of various sources of error on the final output of the LWR model on large networks. The difference between two solutions was evaluated by means of the Wasserstein distance between the two. 
The Wasserstein distance is indeed confirmed to be the right notion of distance to be used in this context, being able to catch the differences among distributions as natural intuition would suggest. In particular, \textit{it is the only distance which is able to quantify the difference between segregated distributions}.

Unfortunately, the numerical approximation of the Wasserstein distance on spaces other than the real line is not trivial. The LP-based method here proposed seems to be appropriate although the computational cost and memory requirements increase nonlinearly with the number of grid nodes. 

Wrong estimation of the position of vehicles at initial time does not seem to have a major impact on the final solution, at least for large times and small networks. Similar conclusions apply to wrong estimation of the fundamental diagram: errors on $\sigma$ or $\fmax$ have approximately the same impact on the final solution, and the discrepancy grows approximately linearly with respect to both $|\sigma^\demand-\sigma^\supply|$ and $|\fmax^\demand-\fmax^\supply|$. 

Conversely, wrong estimation of traffic distribution at junctions and road closures seem to have a far greater impact. Vehicles are redirected in the wrong direction at every passage across the junction, therefore the error grows in time.

Numerical investigation also shows that, in general, the sensitivity grows with the network size. Therefore we expect that LWR previsions based on real data become rapidly unusable on large networks.

\section*{Acknowledgments}
Authors want to thank Sheila Scialanga for the help in writing the numerical code and testing other than Wasserstein distances. 
Authors also thank Fabio Camilli and Simone Cacace for the useful discussions and wrong suggestions.

\appendix
\section{\REV{The numerical scheme for a $2\times 2$ junction [NEW SECTION]}}
We consider the case of a single vertex $\vertex$ with two incoming edges $\edge_1, \edge_2$ and two outgoing edges $\edge^\prime_1,\edge^\prime_2$. 

The total densities on $\edge_i$ and $\edge^\prime_i$, $i=1,2$, are defined by
\begin{equation}\label{App1}
\begin{array}{c}
\rho_{\edge_1,\Nceuno} := \mu_{\edge_1,\Nceuno}((\edge_1,\edge^\prime_1)),\vertex) + \mu_{\edge_1,\Nceuno}((\edge_1,\edge^\prime_2),\vertex),
\smallskip\\
\rho_{\edge_2,\Ncedue} := \mu_{\edge_2,\Ncedue}((\edge_2,\edge^\prime_1)),\vertex) + \mu_{\edge_2,\Ncedue}((\edge_2,\edge^\prime_2)),\vertex),
\smallskip\\
\rho_{\edge^\prime_1,1} := \mu_{\edge^\prime_1,1}((\edge_1,\edge^\prime_1)),\vertex) + \mu_{\edge^\prime_1,1}((\edge_2,\edge^\prime_1)),\vertex),
\smallskip\\
\rho_{\edge^\prime_2,1} := \mu_{\edge^\prime_2,1}((\edge_1,\edge^\prime_2)),\vertex) + \mu_{\edge^\prime_2,1}((\edge_2,\edge^\prime_2)),\vertex).
\end{array}
\end{equation}
Slightly simplifying the notation, system \eqref{MPloc_schema_before} is explicitly written as 
\footnotesize
\begin{equation*}
\begin{array}{c}
\mu_{\edge_1,\Nceuno}^{n+1}(\edge_1,\edge^\prime_1)=\mu_{\edge_1,\Nceuno}^n(\edge_1,\edge^\prime_1)-
\frac{\Dt}{\Dx}\Bigg( \frac{\mu_{\edge_1,\Nceuno}^n(\edge_1,\edge^\prime_1)}{\rho_{\edge_1,\Nceuno}^n} 
G\big(\rho_{\edge_1,\Nceuno}^n,\rho_{\edge^\prime_1,1}^n\big)- 
\alpha^\vertex_{\edge_1\edge^\prime_1} 
G\big(\rho_{\edge_1,\Nceuno-1}^n,\rho_{\edge_1,\Nceuno}^n\big)\Bigg),
\smallskip\\
\mu_{\edge_1,\Nceuno}^{n+1}(\edge_1,\edge^\prime_2)=\mu_{\edge_1,\Nceuno}^n(\edge_1,\edge^\prime_2)-
\frac{\Dt}{\Dx}\Bigg( \frac{\mu_{\edge_1,\Nceuno}^n(\edge_1,\edge^\prime_2)}{\rho_{\edge_1,\Nceuno}^n} 
G\big(\rho_{\edge_1,\Nceuno}^n,\rho_{\edge^\prime_2,1}^n\big)- 
\alpha^\vertex_{\edge_1\edge^\prime_2} 
G\big(\rho_{\edge_1,\Nceuno-1}^n,\rho_{\edge_1,\Nceuno}^n\big)\Bigg),
\smallskip\\
\mu_{\edge_2,\Ncedue}^{n+1}(\edge_2,\edge^\prime_1)=\mu_{\edge_2,\Ncedue}^n(\edge_2,\edge^\prime_1)-
\frac{\Dt}{\Dx}\Bigg( \frac{\mu_{\edge_2,\Ncedue}^n(\edge_2,\edge^\prime_1)}{\rho_{\edge_2,\Ncedue}^n} 
G\big(\rho_{\edge_2,\Ncedue}^n,\rho_{\edge^\prime_1,1}^n\big)- 
\alpha^\vertex_{\edge_2\edge^\prime_1} 
G\big(\rho_{\edge_2,\Ncedue-1}^n,\rho_{\edge_2,\Ncedue}^n\big)\Bigg),
\smallskip\\
\mu_{\edge_2,\Ncedue}^{n+1}(\edge_2,\edge^\prime_2)=\mu_{\edge_2,\Ncedue}^n(\edge_2,\edge^\prime_2)-
\frac{\Dt}{\Dx}\Bigg( \frac{\mu_{\edge_2,\Ncedue}^n(\edge_2,\edge^\prime_2)}{\rho_{\edge_2,\Ncedue}^n} 
G\big(\rho_{\edge_2,\Ncedue}^n,\rho_{\edge^\prime_2,1}^n\big)- 
\alpha^\vertex_{\edge_2\edge^\prime_2} 
G\big(\rho_{\edge_2,\Ncedue-1}^n,\rho_{\edge_2,\Ncedue}^n\big)\Bigg);
\end{array}
\end{equation*}
\normalsize
while system \eqref{MPloc_schema_after} is explicitly written as
\footnotesize
\begin{equation*}
\begin{array}{c}
\mu_{\edge^\prime_1,1}^{n+1}(\edge_1,\edge^\prime_1)=\mu_{\edge^\prime_1,1}^n(\edge_1,\edge^\prime_1)-
\frac{\Dt}{\Dx}\Bigg( \frac{\mu_{\edge^\prime_1,1}^n(\edge_1,\edge^\prime_1)}{\rho_{\edge^\prime_1,1}^n} 
G\big(\rho_{\edge^\prime_1,1}^n,\rho_{\edge^\prime_1,2}^n\big)- 
\frac{\mu_{\edge_1,\Nceuno}^n(\edge_1,\edge^\prime_1)}{\rho_{\edge_1,\Nceuno}^n} 
G\big(\rho_{\edge_1,\Nceuno}^n,\rho_{\edge^\prime_1,1}^n\big)\Bigg),
\smallskip\\
\mu_{\edge^\prime_1,1}^{n+1}(\edge_2,\edge^\prime_1)=\mu_{\edge^\prime_1,1}^n(\edge_2,\edge^\prime_1)-
\frac{\Dt}{\Dx}\Bigg( \frac{\mu_{\edge^\prime_1,1}^n(\edge_2,\edge^\prime_1)}{\rho_{\edge^\prime_1,1}^n} 
G\big(\rho_{\edge^\prime_1,1}^n,\rho_{\edge^\prime_1,2}^n\big)- 
\frac{\mu_{\edge_2,\Ncedue}^n(\edge_2,\edge^\prime_1)}{\rho_{\edge_2,\Ncedue}^n} 
G\big(\rho_{\edge_2,\Ncedue}^n,\rho_{\edge^\prime_1,1}^n\big)\Bigg),
\smallskip\\
\mu_{\edge^\prime_2,1}^{n+1}(\edge_1,\edge^\prime_2)=\mu_{\edge^\prime_2,1}^n(\edge_1,\edge^\prime_2)-
\frac{\Dt}{\Dx}\Bigg( \frac{\mu_{\edge^\prime_2,1}^n(\edge_1,\edge^\prime_2)}{\rho_{\edge^\prime_2,1}^n} 
G\big(\rho_{\edge^\prime_2,1}^n,\rho_{\edge^\prime_2,2}^n\big)- 
\frac{\mu_{\edge_1,\Nceuno}^n(\edge_1,\edge^\prime_2)}{\rho_{\edge_1,\Nceuno}^n} 
G\big(\rho_{\edge_1,\Nceuno}^n,\rho_{\edge^\prime_2,1}^n\big)\Bigg),
\smallskip\\
\mu_{\edge^\prime_2,1}^{n+1}(\edge_2,\edge^\prime_2)=\mu_{\edge^\prime_2,1}^n(\edge_2,\edge^\prime_2)-
\frac{\Dt}{\Dx}\Bigg( \frac{\mu_{\edge^\prime_2,1}^n(\edge_2,\edge^\prime_2)}{\rho_{\edge^\prime_2,1}^n} 
G\big(\rho_{\edge^\prime_2,1}^n,\rho_{\edge^\prime_2,2}^n\big)- 
\frac{\mu_{\edge_2,\Ncedue}^n(\edge_2,\edge^\prime_2)}{\rho_{\edge_2,\Ncedue}^n} 
G\big(\rho_{\edge_2,\Ncedue}^n,\rho_{\edge^\prime_2,1}^n\big)\Bigg).
\end{array}
\end{equation*}
\normalsize
To complete the computation, we sum the sub-densities following \eqref{App1}. Recalling that we have $\alpha_{\edge_1\edge^\prime_1}+\alpha_{\edge_1\edge^\prime_2}=1$ and $\alpha_{\edge_2\edge^\prime_1}+\alpha_{\edge_2\edge^\prime_2}=1$, we get 
\begin{multline*}
\rho_{\edge_1,\Nceuno}^{n+1} = \rho_{\edge_1,\Nceuno}^{n}-
\frac{\Dt}{\Dx}\Bigg( \frac{\mu_{\edge_1,\Nceuno}^n(\edge_1,\edge^\prime_1)G\big(\rho_{\edge_1,\Nceuno}^n,\rho_{\edge^\prime_1,1}^n\big) + \mu_{\edge_1,\Nceuno}^n(\edge_1,\edge^\prime_2)G\big(\rho_{\edge_1,\Nceuno}^n,\rho_{\edge^\prime_2,1}^n\big)}{\rho_{\edge_1,\Nceuno}^n} \\
 - G\big(\rho_{\edge_1,\Nceuno-1}^n,\rho_{\edge_1,\Nceuno}^n\big)\Bigg),
 \end{multline*}
\begin{multline*}
\rho_{\edge^\prime_1,1}^{n+1} = \rho_{\edge^\prime_1,1}^{n}-
\frac{\Dt}{\Dx}\Bigg( G\big(\rho_{\edge^\prime_1,1}^n,\rho_{\edge^\prime_1,2}^n\big)-  
\frac{\mu_{\edge_1,\Nceuno}^n(\edge_1,\edge^\prime_1)G\big(\rho_{\edge_1,\Nceuno}^n,\rho_{\edge^\prime_1,1}^n\big)}{\rho_{\edge_1,\Nceuno}^n} \\
-
\frac{\mu_{\edge_2,\Ncedue}^n(\edge_2,\edge^\prime_1)G\big(\rho_{\edge_2,\Ncedue}^n,\rho_{\edge^\prime_1,1}^n\big)}{\rho_{\edge_2,\Ncedue}^n} 
\Bigg),
\end{multline*}
and analogous expressions for $\rho_{\edge_2,\Ncedue}^{n+1}$ and $\rho_{\edge^\prime_2,1}^{n+1} $.

\end{document}